\documentclass[a4paper,leqno,12pt,draft]{amsart}
\usepackage{amsmath}
\usepackage{amssymb}
\usepackage{amsfonts}
\usepackage{amsthm}
\usepackage{enumerate}
\usepackage{comment}
\usepackage{cite}
\usepackage[dvips]{graphicx,color}%色付け
\usepackage{delarray}%行列式
\usepackage{ascmac}%枠を付ける
\usepackage{fancybox}%影を付けるなど
\usepackage{lineno}
\usepackage[hidelinks,draft=false]{hyperref}
\theoremstyle{plain}
\newtheorem{theorem}{Theorem}[section]
\newtheorem{corollary}[theorem]{Corollary}

\newtheorem{proposition}[theorem]{Proposition}
\newtheorem{lemma}[theorem]{Lemma}

\numberwithin{equation}{section}
\numberwithin{equation}{section}

\def\XXint#1#2#3{{\setbox0=\hbox{$#1{#2#3}{\int}$}
\vcenter{\hbox{$#2#3$}}\kern-.5\wd0}}

\begin{document}
\title[Classification of bifurcation diagrams]{Classification of bifurcation diagrams for semilinear elliptic equations in the critical dimension}
\author{Kenta Kumagai}
\address{Department of Mathematics, Tokyo Institute of Technology}
\thanks{This work was supported by JSPS KAKENHI Grant Number 23KJ0949}
\email{kumagai.k.ah@m.titech.ac.jp}
\date{\today}

\begin{abstract}
We are interested in the global bifurcation diagram of radial solutions for the
Gelfand problem with the exponential nonlinearity and a radially symmetric weight $0<a(|x|)\in C^2(\overline{B_1})$ in the unit ball. When the weight is constant, it is known that the bifurcation curve has infinitely many turning points if the dimension $N\le 9$, and it has no turning points if $N\ge 10$.

In this paper, we show that the perturbation of the weight does not affect the bifurcation structure when $N\le 9$. Moreover,
we find specific radial singular solutions with specific weights and study the Morse index of the solutions. As a consequence,
we prove that the perturbation affects the bifurcation structure in the critical dimension $N=10$. Moreover, we give an optimal classification of the bifurcation diagrams in the critical dimension.

\end{abstract}

\maketitle

\raggedbottom

\section{Introduction}
Let $N\ge 3$ and $B_1\subset \mathbb{R}^N$ be the unit ball. We are interested in the global bifurcation diagram for radial solutions for the semilinear elliptic problem
\begin{equation}
\label{gelfand}
\left\{
\begin{alignedat}{4}
 -\Delta u&=\lambda  a(|x|)e^{u}&\hspace{2mm} &\text{in } B_1,\\
u&>0  & &\text{in } B_1, \\
u&=0  & &\text{on } \partial B_1,
\end{alignedat}
\right.
\end{equation}
where $\lambda>0$ is a parameter and $a:[0,1]\to \mathbb{R}$ satisfies the following
\begin{equation}
a(|x|)\in C^2(\overline{B_1}), \hspace{2mm} a(r)> 0,\hspace{2mm}\text{and}\hspace{2mm} a(0)=1.
\tag{A}
\end{equation}
%generalized Gelfand problem and it is closely related to various physics and geometry problems, such us 
%Gaussian curvature prescription problem (see \cite{chang, chengni}), Edington’s equation 
%(see \cite{Lini}), Chern-Simons theory (see \cite{Yang}), and so on. Then, it was studied by many authors.

\subsection{The case of \texorpdfstring{$a=1$}{Lg}}
\label{a1}
In this case, we remark that each solution of \eqref{gelfand} is radially symmetric and $\lVert u \rVert_{L^{\infty}(B_1)}=u(0)$ by the symmetric result of Gidas, Ni, and Nirenberg \cite{Gidas}. Thus we can use ODE techniques and as a result,
%We change variables to $\overline{u}(r)=u(\sqrt{\lambda}r)$. Then, we can eliminate $\lambda$ in the equation. Thanks to this technique and the implicit function theorem,
we verify that the solution set of \eqref{gelfand} is an unbounded curve described as $(\lambda(\alpha), u(r, \alpha))$ and emanating from $(0,0)$, where $u(r,\alpha)$ is the solution satisfying $\lVert u \rVert_{L^{\infty}(B_1)}=\alpha$. See \cite{korman, Mi2014, Mi2015} for example. We call this set $\{(\lambda(\alpha),\alpha); \alpha>0\}$ the bifurcation curve. A celebrated result of Joseph and Lundgren \cite{JL} states that the bifurcation structure changes depending on $N$. More precisely, they proved that if $N\le 9$, the curve turns infinitely many times around some $\lambda_*$. We call this property of the curve Type I. On the other hand, if $N\ge 10$, the curve can be parametrized by $\lambda$. In addition, $\alpha (\lambda)$ is increasing and it blows up at some $\lambda_*$. We call this property of the curve Type II.
Motivated by their result, a number of attempts have been made in order to determine the bifurcation diagrams for 
various nonlinearities $f>0$. We refer to \cite{Mi2015, Mi2014, Marius, KiWei, Mi2018} for \eqref{gelfand} and \cite{Guowei, Nor, Flore,chend} for related problems. Among them, we note that Miyamoto \cite{Mi2014} provided an example of $f>0$ for which the bifurcation curve of \eqref{gelfand} has at least one but finitely many turning points for a sufficiently large $N\ge 10$. We call this property of the curve Type III.
%Chen and  D\'avila \cite{Davila} showed that the bifurcation branch turns at least one but finitely many turning points when $f(u)=e^u-1$ and $N\ge 10$. 

Along with the studies of the bifurcation diagrams, many studies on the properties of radial singular solutions have been done for various supercritical nonlinearities $f\ge 0$ (see \cite{Mi2014,Merle, Lin, Mi2015, Marius, Mi2018, KiWei, Guowei, Luo, chen,chend,
Mi2020} for details). Here, we say that $(\lambda_{*},U_{*})$ is a radial singular solution of \eqref{gelfand} if $U_{*}(r)$ is a regular solution of \eqref{gelfand} in $(0,1]$ satisfying $U_{*}(r)\to\infty$ as $r\to 0$. In addition to the above studies, Miyamoto and Naito \cite{Mi2023} showed the following properties for a larger class of $f$: there exists the unique radial singular solution $(\lambda_{*}, U_{*})$ such that the bifurcation curve converges to $(\lambda_{*}, U_{*})$. Moreover, they studied the asymptotic behavior of $U_* (r)$ near $r=0$. We point out that the stability of radial singular solutions is deeply connected to the bifurcation structure. More precisely, when $f>0$ is a nondecreasing and convex function in the above class, it follows that the bifurcation diagram is of Type II if and only if the radial singular solution is stable (see \cite{Br, Mi2023} and Section \ref{Hardy sec} for details). Here, we say that a solution $(\lambda,u)$ of \eqref{gelfand} is stable if the linearized operator $-\Delta-\lambda a(|x|) e^{u}$ is nonnegative in the sense that
\begin{align}
\label{stab}
Q_{u}(\xi):=\int_{B_1}|\nabla\xi|^2\,dx-\int_{B_1}\lambda a(|x|) e^{u}\xi^2\,dx\ge 0\hspace{6mm}\text{for all $\xi \in C^\infty_0(B_1)$}\notag.
\end{align}
%Indeed, \cite[Theorem 3.1]{Br} and the above properties of the radial singular solution imply that the bifurcation diagram is of Type II if and only if the radial singular solution is stable (see Section \ref{Hardy sec} for details). 
In particular, when $f(u)=e^u$, we remark that the radial singular solution $(2(N-2), -2\log |x|)$ exists and the solution is stable if and only if $N\ge 10$ by the Hardy inequality. Also from the above facts, we can confirm that the bifurcation structure changes at the dimension $10$. Furthermore, Miyamoto \cite{Mi2014}
focused on the the Morse index $m(U_{*})$ of the radial singular solution $(\lambda_{*},U_{*})$ in the space of radial functions and expected that $m(U_{*})$ is equal to the number of turning points of the bifurcation curve. Here, we explain the definition of $m(U_{*})$ precisely: $m(U_{*})$ is the maximal dimension of a subspace $X\subset H^1_{0,\mathrm{rad}}(B_1)$ such that $Q_{U_{*}}(\xi)<0$ for all $\xi\in X\setminus \{0\}$, where $H^1_{0, \mathrm{rad}}(B_1)$ is the space of radially symmetric functions in $H^1_{0}(B_1)$. We remark that $m(U_{*})=0$ if $U_{*}$ is stable.

\subsection{Weighted case} In this case, in general, we cannot know whether $\lambda$ is parametrized by $\alpha:=\lVert u \rVert_{L^{\infty}(B_1)}$. On the contrary, we prove the following theorem by 
%that the radial solution set is a unbounded analytic curve emanating from $(0,0)$ by 
using the specific change of variables noted in Subsection \ref{ababa}.
\begin{theorem}
\label{diagramth}
Assume that $N\ge 3$ and $a(r)$ satisfies (A). Then, the radial solution set is described as
\begin{equation*}
\{\left(\lambda(\beta), u(r, \alpha(\beta))\right); \beta\in\mathbb{R}\}\hspace{4mm}\text{with}\hspace{4mm}\alpha(\beta)=\beta-\log \lambda(\beta),
\end{equation*}
where $\alpha(\beta):=\lVert u \rVert_{L^{\infty}(B_1)}$. Moreover, the branch
$\mathcal{C}:=\{(\lambda(\beta), \alpha(\beta)); \beta\in \mathbb{R}\}$ is an unbounded analytic curve emanating from $(0,0)$ and there exists $0<\lambda^{*}<\infty$ such that $\lambda(\beta)\le\lambda^{*}$ for all $\beta$. We call this curve the bifurcation curve and we say that $(\lambda(\beta),\alpha(\beta))$ is a turning point
if $\lambda(\beta)$ is a local minimum or local maximum.
\end{theorem}

%When $a$ is general, it is a more challenging problem to study the global bifurcation diagram of radial solutions. We explain the difficulty. As mentioned in Subsection \ref{a1}, when $a=1$, the specific change of variables enables us to know that Moreover, we point out that this change of variables works well in deriving various properties associated with radial singular solutions when $a=1$. On the contrary, when $a$ is general, the change of variables makes change the part of $a$ in the equation. We overcome the difficulty to find another specific change of variables noted in Subsection \ref{ababa}. As a consequence, w

This change of variables enables us to apply the various methods established for the case where $a=1$. As a result, we prove
\begin{theorem}
    \label{singularth}
    Assume that $N\ge 3$ and $a(r)$ satisfies (A). Then, there exists the unique singular radial solution $(\lambda_{*}, U_{*})$ of \eqref{gelfand} in the sense that if $(\lambda, U)$ is a singular radial solution of \eqref{gelfand}, then $\lambda=\lambda_{*}$, $U=U_{*}$. In addition, we have
    \begin{equation}
    \label{convergence schme}
        \lambda(\beta)\to \lambda_{*} \hspace{4mm} \text{and}\hspace{4mm} u(r,\alpha(\beta))\to U_{*} \hspace{2mm} \text{in $C^2_{\mathrm{loc}}(0,1]$\hspace{4mm} as $\beta\to \infty$}.
    \end{equation}
    In particular, it follows that $\alpha(\beta)\to \infty$ if and only if $\beta\to \infty$. Moreover, there exists $\hat{\lambda}>0$ such that the bifurcation curve $\mathcal{C}$ is parametrized by $\lambda$ if $\lambda<\hat{\lambda}$.
    Furthermore, $U_*\in H^1_{0}(B_1)$ and it satisfies 
    \begin{equation}
    \label{asymptotic behavior}
        U_{*}(r)\to - 2\log r + \log2(N-2)-\log \lambda_{*} \hspace{4mm}\text{as $r\to 0$}.
    \end{equation}
\end{theorem}

The following theorem indicates that the perturbation of $a$ does not affect the bifurcation structure if $N\le 9$.

%When $N\le 9$, we can show the following theorem by using the idea of \cite{Mi2014}.
 \begin{theorem}
\label{39th}
    Let $3\le N\le 9$ and we assume that $a(r)$ satisfies $(A)$. Then, the bifurcation diagram is of Type I. 
    %In particular, \eqref{gelfand} has infinity many solutions for $\lambda=\lambda_*$, 
    %Then, 
    %$\lambda(\beta)$ oscillates around $\lambda_*$ as $\alpha(\beta)\to\infty$, where
    %$\alpha(\beta)$ and $\lambda(\beta)$ are that in Theorem \ref{diagramth} and $\lambda_{*}$ is that in Theorem \ref{singularth}. Thus, the bifurcation diagram is of Type I.
    %In particular, \eqref{gelfand} has infinity many solutions for $\lambda=\lambda_*$.
\end{theorem}
%Theorem \ref{39th} indicates that the perturbation of $a$ does not affect the structure of the bifurcation diagram if $N\le 9$.

\subsection{Main results} Motivated from the above fact, we arrive at a fundamental question: does not the perturbation of $a(r)$ make a change to the bifurcation structure in the critical dimension $N=10$? 
%In order to answer this question, we need to study the stability of radial singular solutions. In particular, we need to verify the global structure of the singular solutions, which is difficult if $a$ is general. In this paper,
In order to answer this question, we find specific radial singular solutions for specific weighted terms. By studying the stability and the Morse index of the solutions, we show that, for $N=10$, the bifurcation curve exhibits different types depending on the choice of weight $a(r)$.
%we give a negative answer to the question. 
\begin{theorem}
\label{singularthm}
Let $3\le N\le 10$ and $h>-2(N-2)$. We define 
\begin{equation}
   a_h(r):= \left(1+\frac{h}{2(N-2)}r^2\right)e^{\frac{h}{2N}r^2}\hspace{4mm} and \hspace{4mm} \lambda_{h}=2(N-2)e^{-\frac{h}{2N}}.\notag
\end{equation}
Then, 
\begin{equation}
     U_h(x):=\frac{h}{2N} -2 \log|x|-\frac{h}{2N}|x|^2\notag
\end{equation}
is a radial singular solution of (\ref{gelfand}) for $a=a_h$ and $\lambda=\lambda_{h}$.
Moreover, 
\begin{itemize}
    \item $m(U_h)=\infty$ if $N\le 9$,
    \item $m(U_h)=0$ and $U_h$ is stable if $N=10$ and $h\le H$,
    \item  $1\le m(U_h)<\infty$ if $N=10$ and $h>H$,
\end{itemize}
where $H$ is the first eigenvalue of the Laplacian in the unit ball in $N=2$.

As a consequence, when $N=10$ and $a=a_h$, the bifurcation diagram is of Type II if and only if $h\le H$. In particular, the bifurcation structure changes at $a_H$ when $N=10$.   
\end{theorem}
We remark that the exponent $H$ arises from the best constant of the %KANSI
improved Hardy inequality \cite[Theorem 4.1]{Br} and this inequality plays a key role in studying the stability/instability of the singular solutions. In addition, thanks to Theorem \ref{singularthm} and \cite[Conjecture 1.4]{Mi2014}, we can expect that the bifurcation diagram is of Type III for $a=a_h$ if $N=10$ and $h>H$. In the following theorem, we show that this expectation is true. Moreover, we give an optimal classification for the bifurcation diagrams in the critical dimension.
\begin{theorem}
    \label{Mainthm}
Let $N=10$ and we assume that $a(r)$ satisfies $(A)$. Then, the bifurcation diagram is of:
\begin{itemize}
\item[(i)] Type II if $(a/a_H)'\le 0$ in $(0,1]$,
\item[(ii)] Type III if $(a/a_H)'>0$ in $(0,1]$,
\end{itemize}
where $H$ and $a_H$ are those in Theorem \ref{singularthm}.
In particular, when $N=10$ and $a=a_h$, the bifurcation diagram is of Type II if $h\le H$ and of Type III if $h>H$. 
\end{theorem}
We recall that the classical Hardy inequality controls the bifurcation structure when $a=1$. On the other hand, Theorem \ref{Mainthm} tells us that the improved Hardy inequality controls it when the weight $a$ is general and $N=10$. Therefore, our results state not only the important difference for the bifurcation structure between the case of $a=1$ and the weighted case, but also the reason why the difference happens. We also remark that it is a challenging problem to show that the bifurcation diagram is of Type III. In addition, to the best of the author's knowledge, the bifurcation of type III was  only confirmed in the case where $a=1$ and a cleverly chosen $f$. We not only provide many examples of Type III bifurcation but also give a new method to prove that the bifurcation diagram is of Type III. Finally, we remark that we can prove the change of the bifurcation structure when $N\ge 11$ and a similar classification holds when $N=11$ or $a=a_h$ by using similar methods.

\subsection{Application: A change of the regularity property of extremal solutions.} 
\label{applisub}
In this subsection, we focus on the minimal branch of the bifurcation curve. It is well known \cite{Ali,B,Br,Dup} that the set of minimal solutions is a continuous curve emanating from $(0,0)$ and 
there exists $\lambda^* \in(0,\infty)$ such that
\begin{itemize}
\item For $0<\lambda<\lambda^*$, there exists a minimal classical solution $u_\lambda\in C^2(\overline{B_1})$. In particular, $u_\lambda$ is radially symmetric, stable, and $u_\lambda<u_{\lambda'}$ for $\lambda<\lambda'$.
\item For $\lambda=\lambda^*$, we define $u^*:=\lim_{\lambda \uparrow \lambda^* } u_\lambda$. Then, $u^*$ is called the extremal solution, which is a
radially symmetric stable weak solution in the sense that $u^*\in L^1(B_1)$, $a(|x|)e^{u^*}\mathrm{dist}(\cdot,\partial \Omega)\in L^1(B_1)$, and 
\begin{equation}
-\int_{B_1}u^*\Delta \xi \,dx = \int_{B_1}\lambda^*a(|x|)e^{u^*}\xi \,dx\hspace{8mm} \text{for all}\hspace{2mm}\xi\in C^2_0(\overline{B_1}).\notag
\end{equation}
\item For $\lambda>\lambda^*$, there exists no weak solution.
\end{itemize}
We remark that when the nonlinearity $f>0$ is nondecreasing and superlinear (i.e., $f(u)/u \to \infty$ as $u\to\infty$), the above properties hold even if the 
weight $a$ and the domain $\Omega$ are arbitrary.\footnote{To avoid a misunderstanding, we remark that $u$ is not always radially symmetric if $a$ is not radially symmetric or $\Omega$ is not a ball.}  
Thus, the boundedness of extremal solutions is well-studied instead of the bifurcation diagrams. We refer to \cite{CR,Ye, A2016,AC,Davdup,sansan, Dup, Ned,cabre2017,cabre2019,cabre2010,CCS,CR-O,
Vil,cabrecapella,CFRS}. 
As a consequence of Theorem \ref{Mainthm}, we get
%these properties hold and thus we can consider extremal solutions even if the 
%weight $a$, the nonlinearity $f$, the domain $\Omega$ are arbitrary.\footnote{To avoid a misunderstanding, we remark that $u$ is not always radially symmetric if $a$ is not radially symmetric or $\Omega$ is not a ball.}  
%Thus, it is well-studied the boundedness of extremal solutions instead of the bifurcation diagram. We refer to \cite{CR,Ye, A2016,AC,Davdup,sansan, Dup, Ned,cabre2017,cabre2019,cabre2010,CCS,CR-O,
%Vil,cabrecapella,CFRS}. 
%As a consequence of Theorem \ref{Mainthm}, we get
\begin{corollary}
\label{maincor}
Let $3\le N\le10$ and we assume that $a(r)$ satisfies $(A)$. Then, the extremal solution of $(\ref{gelfand})$ is:
\begin{itemize}
    \item[(i)] bounded if $N\le 9$,
    \item[(ii)] singular if $(a/a_H)'\le 0$ in $(0,1]$ and $N=10$,
    \item[(iii)] bounded if $(a/a_H)'>0$ in $(0,1]$ and $N=10$,
\end{itemize}
where $H$ and $a_H$ are those in Theorem \ref{singularthm}.
\end{corollary}
We introduce known results in the specific case of \eqref{gelfand}. It is known by the method in \cite{CR} that $u^*$ is bounded when $N\le 9$. Moreover, when $N\ge 10$, Bae \cite{Bae} proved that $u^*$ is singular if
\begin{equation}
\label{bae1}
    a(r)\le \frac{N-2}{8}\inf_{0\le s \le r} a(s) \hspace{4mm}\text{for $r>0$}.
\end{equation}
We note that the condition (\ref{bae1}) means that $a(r)$ is nonincreasing when $N=10$. Thus, Corollary \ref{maincor} is not only a generalization of known results when $N\le 10$ but also the first result which states that the boundedness property of $u^*$ changes depending on the perturbation of $a$. Moreover, we remark that when $a=1$ and $\Omega$ is general, Crandall and Rabinowitz \cite{CR} proved that $u^*$ is bounded if $N\le 9$. Based on this result, Brezis and V\'azquez asked in \cite{Br} whether $u^*$ is always singular for all $\Omega$ if $N\ge 10$ and $a=1$. This expectation was answered negatively: some examples
of $\Omega$ were constructed which $u^*$ is bounded for some $N\ge 10$. We refer to \cite{Davdup, AC}. Corollary \ref{maincor} describes a similar property to the above fact.
\subsection{Structure of the paper}
In Section \ref{basic}, we combine the specific change of variables with the method established for the case when $a=1$ and as a result, we prove Theorem \ref{diagramth} and Theorem \ref{singularth}. In Section \ref{Hardy sec},
we prove the optimality\footnote{We explain the meaning of optimality in Section \ref{Hardy sec}.} of the improved Hardy inequality. By using it, we study the stability and the Morse index of $U_h$ and prove Theorem \ref{singularthm}. In Section \ref{separate sec}, we prove the comparison results. The results work well to prove Theorem \ref{Mainthm}. In Section \ref{mainsec},
we prove Theorem \ref{39th} by using the idea of \cite{Mi2014} and we prove Theorem \ref{Mainthm}.
\section{Basic properties of radial solutions}
\label{basic}
In this section, we prove basic properties of radial solutions and as a result, we prove Theorem \ref{diagramth} and Theorem \ref{singularth}. We first fix the notation.
In this paper, we only deal with radial solutions. Thus, from now on, we identify the the operator $\Delta$ with the operator $\frac{d^2}{dr^2}+\frac{N-1}{r}\frac{d}{dr}$, and $r$ with $|x|$. 

\subsection{Preliminaries}
In this subsection, we prove some basic properties for radial solutions of 
\begin{equation}
\label{apriorieq}
-\Delta u = \lambda a(r) e^u, \hspace{4mm}0<r<R,
\end{equation}
where $R>0$ and $\lambda$ are positive constants and $a:[0,R]\to \mathbb{R}$ satisfies the following
\begin{equation}
a(|x|)\in C^2(\overline{B_R}), \hspace{2mm} a(r)> 0,\hspace{2mm}\text{and}\hspace{2mm} a(0)=1.
\tag{$\hat{A}$}
\end{equation}
We note that the following results in this subsection are proved as a modification of \cite{Mi2020}. We begin by introducing an a priori estimate of positive radial solutions.
\begin{lemma}
\label{apriorilemma}
Let $0<\lambda_0<\lambda_1$ and let $u$ be a positive radial solution of \eqref{apriorieq} for some $\lambda_0<\lambda<\lambda_1$. We assume that $a(r)$ satisfies $(\hat{A})$. Then,
there exist $C_1>0$ and $C_2>0$ depending only on $\lambda_0, \lambda_1, a(r), R$ such that 
\begin{equation}
\label{aprioriestimate 2-1}
    u(r)\le -2\log r + C_1, \hspace{4mm} 0\le -\frac{d}{dr}u(r) \le \frac{C_2}{r} \hspace{4mm} \text{for $0<r\le R$},
\end{equation}
and
\begin{equation}
    \label{aprioriestimate 2-2}
    -r^{N-1}\frac{d}{dr}u(r)= \lambda\int_{0}^{r}s^{N-1}a(s)e^{u(s)}\,ds \hspace{4mm} \text{for $0<r\le R$}.
\end{equation}
\end{lemma}
We remark that we can obtain the a priori estimate in the same way as in \cite[Lemma 2.3]{Mi2020} and thus we omit the proof. 

Then, we quote the result in \cite{Mi2020}. Let $t_0\in \mathbb{R}$, $p>0$, and $-\infty\le \alpha<\beta\le \infty$. In addition, we take $H\in C^{0}((\alpha,\beta))$, $G\in C^{0}([t_0,\infty)\times (\alpha,\beta))$ and we assume that there exists $\gamma\in (\alpha,\beta)$ such that 
\begin{equation*}
    (w-\gamma)H(w)>0\hspace{4mm}\text{for all $w\in (\alpha,\beta)\setminus \{\gamma\}$}.
\end{equation*}
We remark that the assumption implies $H(\gamma)=0$. In \cite{Mi2020}, the authors consider the following ODE
\begin{equation}
\label{general}
    \frac{d^2}{dr^2}w-p\frac{d}{dr}w+H(w)+G(t,w)=0 \hspace{4mm}\text{for $t\ge t_0$}
\end{equation}
and they prove the following
\begin{lemma}
\label{ordlem}
Let $w\in C^2([t_0,\infty))$ be a solution of \eqref{general} satisfying $\alpha<w(t)<\beta$ for all $t\ge t_0$. Assume that $w$ satisfies 
    \begin{equation*}
        \alpha<\limsup_{t\to\infty} w(t)<\beta\hspace{4mm}and \hspace{4mm} \lim_{t\to\infty}G(t, w(t))=0. 
    \end{equation*}
    If $\alpha=-\infty$, assume in addition that 
    \begin{equation*}
       \lim_{t\to\infty} \frac{G(t, w(t))}{e^{w(t)}}=0.
    \end{equation*}
    Then, $\lim_{t\to\infty}w(t)=\gamma$.
\end{lemma}
By using Lemma \ref{ordlem}, we study the asymptotic behavior of positive radial singular solutions of \eqref{apriorieq}. 
\begin{proposition}
\label{asymptotic prop}
    Let $u$ be a positive singular radial solution of \eqref{apriorieq}. We assume that $a(r)$ satisfies $(\hat{A})$. Then, $u$ satisfies
    \begin{equation*}
        u(r)\to -2\log r -\log \lambda+ \log2(N-2) \hspace{2mm}\text{as $r\to 0$}. 
    \end{equation*}
\end{proposition}

\begin{proof}

%This proposition is proved as a modification of the proof of \cite[Lemma 2.4]{} and \cite[Lemma ]{}. By reader's convenience, we state the sketch of the proof.
We define  
 \begin{equation*}
 w(t)=u(r)+2\log r + \log \lambda-\log 2(N-2)\hspace{2mm}\text{with $t=-\log r$}. 
 \end{equation*}
 Then, $w$ satisfies 
\begin{equation*}
    \frac{d^2}{dt^2}w-(N-2)\frac{d}{dt}w+2(N-2)(e^w-1) + 2(N-2)(a(e^{-t})-1) e^w=0
\end{equation*}
for $t\ge t_0$ with $t_0:=-\log R$. Moreover, we obtain 
\begin{align}
\label{fact2-1}
    -\infty<\limsup_{r\to 0}(u(r)+2\log r)\le C_1
\end{align}
by using Lemma \ref{apriorilemma} and the same method as in the proof of \cite[Lemma 2.4]{Mi2020}, where $C_1$ is that in \eqref{aprioriestimate 2-1}. The existence of this change of variables and the fact \eqref{fact2-1} allow for us to apply Lemma \ref{ordlem} for $p=N-2$, $H(w)=2(N-2)(e^{w}-1)$, $G(t,w)=2(N-2)(a(e^{-t})-1)e^w$, $\alpha=-\infty$, $\beta=\infty$. As a result, it holds $\lim_{t\to \infty} w(t)=0$ and we get the result.
\end{proof}

\subsection{Proof of Theorem \ref{diagramth} and Theorem \ref{singularth}}
\label{ababa}
 We begin by introducing a specific change of variables and proving Theorem \ref{diagramth}.
We extend $a$ on $[1,\infty)$ a positive function such that $a(r)\in C^2(\mathbb{R})$.
 For any $\beta\in \mathbb{R}$, we define $v=v(r,\beta)$ as the solution of 
\begin{equation}
\label{odev}
\left\{
\begin{alignedat}{4}
 &-\Delta v= a(r)e^{v}, \hspace{14mm}0<r<\infty,\\
&v(0,\beta)=\beta, \hspace{4mm} \frac{d}{dr} v(0,\beta)=0.
\end{alignedat}
\right.
\end{equation}
\begin{proof}[Proof of theorem \ref{diagramth}]
For any $\beta\in \mathbb{R}$, we define $u(r,\alpha(\beta))=v(r, \beta)-\log \lambda(\beta)$ with $\lambda(\beta):=e^{v(1,\beta)}$ and $\alpha(\beta):=\beta- \log \lambda(\beta)$. Then, we can easily confirm that $(\lambda(\beta), u(r,\alpha(\beta)))$ is a solution of \eqref{gelfand}. On the other hand, if $(\lambda, u)$ is a solution of \eqref{gelfand} which satisfies $\alpha=\lVert u \rVert_{L^{\infty}(B_1)}$ for some $\alpha>0$, then $v(r,\beta):=u+\log \lambda$ is a solution of \eqref{odev} for $\beta=\alpha+\log \lambda$. Thus, we have
$(\lambda,u)=(\lambda(\beta), u(r,\alpha(\beta)))$. Therefore, every radial solution is parametrized by $\beta$. 
Moreover, the analyticity of this curve follows from the analyticity of the exponential function. The left part can be proved by using the properties of the minimal branch (see Subsection \ref{applisub}).
\end{proof}

Then, based on the idea of \cite{Lin, Mi2020}, we prove the following
\begin{proposition}
\label{pohopro}
Let $q>\frac{N+2}{N-2}$, $\lambda>0$. We assume that $a(r)$ satisfies $(\hat{A})$ and $u\in C^2(B_R)$ is a positive radial solution of \eqref{apriorieq} satisfying $u(R)=0$.
\begin{itemize}
    \item[(i)] Let $u(0)>q+1$ and we assume that $u(r)\ge q+1$ for $0\le r\le r_0$ with some $r_0>0$. Then,
    \begin{equation*}
        0<-r\frac{du}{dr}(r)<\frac{2N}{q+1}u(r)+ C\hspace{4mm}\text{for $0<r<r_0$,}
    \end{equation*}
    where $C$ is a constant depending only on $a$ and $R$.
    \item[(ii)]Define
    \begin{equation*}
        \tau=\frac{1}{3}\left(1-\frac{2N}{(q+1)(N-2)}\right).
    \end{equation*}
     Take any $\rho\ge q+1+C$ and define $r_{\rho}$ by 
    \begin{equation*}
        r_{\rho}:=\left(\frac{2N\rho}{\lambda e^{\rho/\tau}\lVert a\rVert_{L^{\infty}(B_R)} }\right)^{1/2},
    \end{equation*}
    where $C$ is that in (i).
    If $u(0)>\rho/\tau$ then $u$ satisfies
    \begin{equation*}
        u(r)>\rho \hspace{4mm}\text{for $0\le r \le r_{\rho}$.}
    \end{equation*}
\end{itemize}
\end{proposition}
In order to prove this proposition, we quote a Pohozaev-type identity. This identity is proved by Ni and Serrin \cite{Niserrin} when $a=1$.
\begin{lemma}
\label{generalpohozaev}
Let $u$ be a radial solution of \eqref{apriorieq} and let $\mu$ be an arbitrary constant. We assume that $a(r)$ satisfies $(\hat{A})$.
Then, for any $r\in (0,R)$, we have
\begin{align}
&\frac{d}{dr}\left\{r^N\left(\frac{1}{2}\left(\frac{du}{dr}\right)^2
+\lambda a(r) (e^{u(r)}-1) +\frac{\mu}{r}u(r)\frac{du}{dr}\right)\right\}-\lambda \frac{da}{dr}(e^{u(r)}-1)r^N\notag\\
&=r^{N-1}\left\{N\lambda a(r)(e^{u(r)}-1)-\lambda\mu a(r) u(r)e^{u(r)}+\left(\mu+1-\frac{N}{2}\right)\left(\frac{du}{dr}\right)^2\right\}.\notag
\end{align}
\end{lemma}
\begin{proof}
By a direct computation, we get the result.
\end{proof}

\begin{proof}[Proof of Proposition \ref{pohopro}]
Put $\mu:=N/(q+1)$. Then, we remark that the right-hand side of the inequality in Lemma \ref{generalpohozaev} is negative for $0\le r \le r_0$ since it holds $N(e^{u(r)}-1)<\mu e^{u(r)}u(r)$ for $0\le r\le r_0$. By integrating the inequality in Lemma \ref{generalpohozaev} on $[0,r]$ with $0\le r\le r_0$, we obtain
\begin{equation*}
    \frac{1}{2}r^{N} \left(\frac{du}{dr}\right)^2
+\lambda a(r)r^{N} (e^{u(r)}-1) +\mu r^{N-1}\frac{du}{dr}u<\lambda\int_{0}^{r}s^{N} \frac{da}{ds}(e^{u(s)}-1)\,ds.
\end{equation*}
Moreover, it follows by (\ref{aprioriestimate 2-2}) that 
\begin{align*}
    \lambda\int_{0}^{r} s^{N}\frac{da}{ds}(e^{u(s)}-1)\,ds\le Cr\lambda\int_{0}^{r} s^{N-1}a(s)e^{u(s)}\,ds\le -Cr^{N-1}\frac{du}{dr},
    \end{align*}
where $C>0$ is depending only on $a$ and $R$. Thus, we get (i). As a result, we can prove (ii) by using a method similar to the proof of \cite[Lemma 5.1 (ii)]{Mi2020}.
\end{proof}

\begin{proof}[Proof of Theorem \ref{singularth}.]
We divide the proof into the following three steps.

\vspace{5pt}

\noindent
\textbf{Step 1.} \textit{The bifurcation curve is parametrized by $\lambda$ when $\lambda$ is sufficiently small.}

%We claim that there exists a small $\hat{\lambda}$ such that the bifurcation curve $C$ is parametralized by $\lambda$ if $\lambda<\hat{\lambda}$.
We first claim that there exist $C_1$, $\delta$ depending only on $a$, $N$ such that 
if $\alpha(\beta)>C_1$, then $\lambda(\beta)>\delta$. Indeed, let $q$ and $\rho$ be constants satisfying the assumption in Proposition \ref{pohopro} and we define $C_1:= \rho/\tau$. Let $\delta$ be a small constant such that $\tau_{\rho}\ge1$ if $\lambda<\delta$. Then, thanks to Proposition \ref{pohopro}, if $\lambda(\beta)\le \delta$, then $\alpha(\beta)\le C_1$. Thus, we get the claim. 

Let $\eta:=\lVert u^*\rVert_{L^\infty(B_1)}+\log \lambda^*$, where $u^*$ is the extremal solution with $\lambda=\lambda^*$. We fix a constant $\hat{\lambda}<\min\{\delta,\lambda^{*}\}$ such that $C_1+\log \hat{\lambda}<\eta$. Let $(\lambda(\beta),u(r,\alpha(\beta))$ be a solution of \eqref{gelfand} such that $\lambda(\beta)\le \hat{\lambda}$. Since $\hat{\lambda}<\delta$, we get $\alpha(\beta)\le C_1$. Moreover, since $\beta=\alpha(\beta)+\log \lambda(\beta)$, 
we have $\beta<\eta$. We recall that the bifurcation curve is parametrized by $\beta$. Therefore, the solution $(\lambda(\beta), u(r,\alpha(\beta)))$ belongs to the minimal branch of $C$. Since the minimal branch is parametrized by $\lambda$, we get the result. 

We recall that $\beta=\alpha(\beta)+\log \lambda(\beta)$ and $\lambda(\beta)\le \lambda^{*}$ for all $\beta$. Thus, as a consequence of Step 1, we verify that $\alpha(\beta)\to\infty$ if and only if $\beta\to\infty$.
\vspace{5pt}

\noindent
\textbf{Step 2.} \textit{The radial singular solutions $(\lambda_{*},U_{*})$ are at most one.}

Let $(\lambda^{1}_{*}, U^1_{*})$ and $(\lambda^{2}_{*}, U^2_{*})$ be radial singular solutions. Thanks to Proposition \ref{asymptotic prop}, $(\lambda^{1}_{*}, U^1_{*})$ and $(\lambda^{2}_{*}, U^2_{*})$ satisfy \eqref{asymptotic behavior}. Let $V^{i}_{*}:=U^{i}_{*}+\log \lambda^{i}_{*}$ with $i=1,2$. Then $V_{*}=V^{i}_*$ satisfies
\begin{equation}
\label{singulareq}
\left\{
\begin{alignedat}{4}
 -\Delta V_{*}&= a(r)e^{V_{*}}, &&0<r<\infty,\\
V_{*}(r)&= -2\log r + \log 2(N-2)+o(1) \hspace{6mm} &&\text{as $r\to 0$}.
\end{alignedat}
\right.
\end{equation}
Moreover, we have $V^{i}_*(1)=\log \lambda^{i}_{*}$.
By \cite[Theorem 1.6]{Bae}, it holds $V^{1}_*=V^{2}_*$. In particular, since $V^{1}_*(1)=V^{2}_*(1)$, it holds $\lambda^{1}_{*}=\lambda^{2}_{*}$ and $U_*^{1}=U_*^{2}$. 

\vspace{5pt}

\noindent
\textbf{Step 3} \textit{Existence of the radial singular solution and  \eqref{convergence schme}.} 

Let $\delta>0$ and let $(\lambda(\beta_{k}), u(r,\alpha(\beta_{k})))\in C$ with $\beta_{k}\to\infty$ as $k\to\infty$. Thanks to Step 1, we get $\alpha(\beta_k)\to \infty$ as $k\to\infty$ and $\hat{\lambda}< \lambda(\beta_k) \le \lambda^*$ for all $k$ sufficiently large. Moreover, by Lemma \ref{apriorilemma} and the elliptic regularity theory (see \cite{gil}), there exists some $0<\delta_0<1$ such that it holds 
$\lVert u(r,\alpha(\beta_k))\rVert_{C^{2+\delta_0}(B_1 \setminus B_{\delta})}<C$, where $C$ is a constant independent of $k$. Since $\delta>0$ is arbitrary, by Ascoli-Arzel\'a theorem and a diagonal argument, there exists a subsequence $k_j$ of $k$ and a solution $U_{*}\in C^2_{\mathrm{loc}}((0,1])$ of \eqref{gelfand} for $\lambda=\lambda_{*}$ such that  
$u(r,\alpha(\beta_{k_{j}}))\to U_{*}$ in $C^2_{\mathrm{loc}}((0,1])$ as $j\to\infty$
and $\lambda(\beta_{k_{j}})\to \lambda_{*}$ as $j\to \infty$. Moreover, thanks to Proposition \ref{pohopro}, we verify that $(\lambda_{*}, U_{*})$ is a radial singular solution.
By the uniqueness of the radial singular solution, 
we have $u(r,\alpha(\beta_{k}))\to U_{*}$ as $\beta_k\to\infty$.

We can easily know the left part by using Lemma \ref{apriorilemma} and Proposition \ref{asymptotic prop}.
\end{proof}

\subsection{Some notations}
From now on, we fix some notations. We define 
$v(r,\beta)$ and $V_{*}$ as those in \eqref{odev} and \eqref{singulareq} respectively. Then, we define $v_h(r,\beta)$ and $V_h$ as those in \eqref{odev} and \eqref{singulareq} for $V_{*}=V_{h}$ and $a=a_h$, where $a_h$ is that in Theorem \ref{singularthm}. In particular, $v_0(r,\beta)$ and $V_0$ mean those for $a=a_0=1$ and we remark that $V_0(r)=-2 \log r$. 
Moreover, we define $(\lambda(\beta), u(r,\alpha(\beta)))$ as that in Theorem \ref{diagramth} and we define $a_h, \lambda_h, U_h$ as those in Theorem \ref{singularthm}. Then, we define $u^*$ as the extremal solution for $\lambda=\lambda^*$.

\section{Stability and Morse index of the specific singular solutions}
\label{Hardy sec}
In this section, we study the stability and the Morse index of $U_h$, where $U_h$ is that in 
Theorem \ref{singularthm}. We first recall the improved Hardy inequality.
\begin{proposition}[see \cite{Br}]
\label{improved hardy}
    Let $N\ge 2$ and $\Omega$ be a bounded domain in $\mathbb{R}^{N}$. Then, for every $\xi\in H^{1}_{0}(\Omega)$ we have
    \begin{equation}
    \label{imp}
         \int_{\Omega}|\nabla \xi|^2\,dx\ge \frac{(N-2)^2}{4}\int_{\Omega} \frac{\xi^2}{|x|^2}\,dx+ H \left(\frac{\omega_N}{|\Omega|}\right)^{\frac{2}{N}}\int_{\Omega} \xi^2\,dx,  
    \end{equation}
where $H$ is the first eigenvalue of the Laplacian in the unit ball in $N=2$ and $\omega_N$ is the measure of the $N$-dimensional unit ball.
\end{proposition}
We consider only the case $\Omega=B_1$. We remark that the optimizer of the inequality (\ref{imp}) is $\xi=|x|^{\frac{2-N}{2}}\varphi$, where $\varphi>0$ is a first eigenfunction of the Laplacian in the unit ball in $N=2$ and it does not belong to $H^1_{0}(B_1)$. As mentioned in the introduction, the improved Hardy inequality plays a important role in studying the stability/instability of $U_h$. Thus, it is important to prove the optimality of the best constant $H$ for the test functions in $H^1_{0}(B_1)$. In other words, we find a test function $\hat{\xi}\in H^1_{0}(B_1)$ such that the inequality (\ref{imp}) does not hold for $\xi=\hat{\xi}$ if we replace $H$ with $H+\varepsilon$. To be more precise, we prove
\begin{proposition}
\label{improvedhardy2}
  Let $N\ge 2$ and let $\varepsilon:[0,1]\to [0,1]$ be a nondecreasing function such that $\varepsilon(\frac{1}{2})>0$. Then, there exists $\xi\in H^{1}_{0}(B_1)$ such that
\begin{equation*}
\int_{B_1}|\nabla \xi|^2\,dx< \frac{(N-2)^2}{4}\int_{B_1} \frac{\xi^2}{|x|^2}\,dx+ \int_{B_1} \left(H+\varepsilon (|x|)\right)\xi^2\,dx.\notag
\end{equation*}
\end{proposition}
\begin{proof}
Let $\varepsilon>0$ and $n\in \mathbb{N}$. We define
    \begin{equation}
    \phi_n(x)=\min\{\frac{n}{1-\log |x|}, 1\}, \hspace{4mm} \xi_{n}(x)= \phi_{n}\varphi|x|^{\frac{2-N}{2}}\notag,
    \end{equation}
    where $\varphi>0$ is a first eigenfunction of the Laplacian in the unit ball in $N=2$. Then it holds 
    $\xi_{n}\in H^1_{0}(B_1)$. Moreover, by a direct computation, we have
    \begin{align}
       \int_{B_1}&|\nabla \xi_{n}|^2\,dx-\frac{(N-2)^2}{4}\int_{B_1} \frac{\xi_{n}^2}{|x|^2}\,dx- \int_{B_1} \left(H+\varepsilon(|x|) \right)\xi_{n}^2\,dx\notag\\
&=N\omega_{N}\int_{0}^{1}\left((\phi_{n}\varphi)'\right)^2 r-(N-2)\phi_{n}\varphi(\phi_{n}\varphi)'-\left(H+\varepsilon \right)(\phi_{n}\varphi)^2 r\,dr \notag\\
&=N\omega_{N}\int_{0}^{1}\left((\phi_{n}\varphi)'\right)^2 r -(H+\varepsilon )(\phi_{n}\varphi)^2 r\,dr \notag.
\end{align}
 Let $\delta>0$ be a constant which satisfies 
 \begin{equation}
 \label{yoikanzi}
      H\delta^2 \int^{1}_{0}\varphi^2 r\,dr \le\frac{1}{2} \int^{1}_{0}\varepsilon\varphi^2 r\,dr.
 \end{equation}
 
We note that $0<\phi_{n}\le 1$ for all $0<r<1$ and
$\phi_{n}=1$ for $r>e^{1-n}$. In addition, we note that it holds
\begin{equation*}
    \int_{0}^{1}(\varphi')^2 r\,dr= H\int_{0}^{1}\varphi^2 r\,dr.
\end{equation*}
By Young's inequality, 
 (\ref{yoikanzi}), and the above notation, we get
\begin{align}
&\int_{0}^{1}\left((\phi_{n}\varphi)'\right)^2 r -(H+\varepsilon)(\phi_{n}\varphi)^2 r\,dr \notag\\
&\hspace{2mm}\le(1+\delta^2)\int_{0}^{1} \phi_{n}^2(\varphi')^2 r\,dr-\int_{0}^{1}(H+\varepsilon)(\phi_{n}\varphi)^2 r\,dr+\left(1+\frac{1}{\delta^2}\right)\int_{0}^{1} (\phi_{n}')^2\varphi^2 r\,dr\notag\\
&\hspace{2mm}\le H(1+\delta^2)\int_{0}^{1}\varphi^2 r\,dr-\int_{e^{n-1}}^{1}(H+\varepsilon)\varphi^2 r\,dr+\left(1+\frac{1}{\delta^2}\right)\int_{0}^{1} (\phi_{n}')^2\varphi^2 r\,dr\notag\\
&\hspace{2mm}\le \int_{0}^{e^{1-n}} \left(\left(1+\frac{1}{\delta^2}\right)\frac{n^2}{r(1-\log r)^4}+(H+\varepsilon)r\right)\varphi^2\,dr\notag\\
&\hspace{62mm}+H\delta^2\int_{0}^{1}\varphi^2 r\,dr -\int_{0}^{1}\varepsilon\varphi^2 r\,dr\notag\\
&\hspace{2mm}\le \left(\frac{1}{3n}\left(1+\frac{1}{\delta^2}\right)+\frac{(H+1)e^{2-2n}}{2}\right)\lVert \varphi\rVert_{L^{\infty}(B_1)}^2-\frac{1}{2} \int_{e^{1-n}}^{1}\varepsilon\varphi^2 r\,dr\notag\\
&\hspace{2mm}\to  -\frac{1}{2} \int_{0}^{1}\varepsilon\varphi^2 r\,dr<0\hspace{4mm}\text{as $n\to \infty$.}\notag
\end{align}

Thus, by letting $n$ sufficiently large, the result holds for $\xi=\xi_{n}$.
\end{proof}

In preparation for proving Theorem \ref{singularthm}, we quote the following proposition that describes the relationship between the bifurcation structure and the stability of radial singular solutions.
\begin{proposition}[see \cite{Br,B}]
    \label{brezis V}
Let $(\lambda_{*}, U_{*})$ be a singular solution of (\ref{gelfand}) such that $U_{*}\in H^1_{0}(B_1)$ and $U_{*}$ is stable. Then, $U_{*}$ is the extremal solution for $\lambda_{*}$. In particular, the bifurcation diagram is of Type II.
\end{proposition}
We remark that the authors prove Proposition \ref{brezis V} in \cite{Br} by applying \cite[Corollary 2]{B} when $a=1$. As the same way, we can prove it for the weighted case. Moreover, we remark that the extremal solution of \eqref{gelfand} is radially symmetric and stable (see Subsection \ref{applisub}). Thus, thanks to Theorem \ref{singularth} and this proposition, we verify that the bifurcation diagram is of Type II if and only if the radial singular solution is stable.

\begin{proof}[Proof of Theorem \ref{singularthm}.]
By a direct computation, we verify that $U_{h}$ is the radial singular solution of \eqref{gelfand} for $a=a_{h}$ and $\lambda= \lambda_{h}$. Moreover, we verify that the linearized operator $-\Delta - a_{h} e^{U_h}$ is equal to $-\Delta-2(N-2)|x|^{-2}-h$ and it holds
\begin{align}
\label{qnohyouzi}
  Q_{U_h}(\xi)= 
    \int_{B_1}|\nabla \xi|^2\,dx- 2(N-2)\int_{B_1} \frac{\xi^2}{|x|^2}\,dx- h \int_{B_1} \xi^2\,dx
\end{align}
for all $\xi\in H^1_{0}(B_1)$. We claim that $U_h$ is stable if and only if $m(U_h)=0$. Indeed, we denote by $\xi^*\in H^1_{0}(B_1)$ the symmetric  
rearrangement of $\xi$. Then it follows by Lemma \ref{apenlem1}
that $Q_{U_h}(\xi^*)\le Q_{U_h}(\xi)$. Thus, we get the claim.
\vspace{5pt}

\textbf{Case 1.} \textit{The case of $N\le 9$.}

In this case, we use a similar method used in \cite{Mi2014}.
Let $\varepsilon>0$ be a small constant 
such that 
\begin{equation*}
    \delta:=2(N-2)-\frac{(N-2)^2+\varepsilon^2}{4}>0
\end{equation*}
and we define $\xi_{j}(r)=r^{\frac{2-N}{2}} \sin(\frac{\varepsilon}{2}\log r) \chi_{[r_{j+1},r_j]}$ with $r_j=e^{-2\pi j/\varepsilon}$. Then, it holds $\xi_{j}\in H^1_{0}(B_1)$ and it satisfies
\begin{equation*}
    -\Delta \xi_j=\left(\frac{(N-2)^2+\varepsilon^2}{4}\right)r^{-2} \xi_j\hspace{4mm}\text{for all $r\in (r_{j+1},r_j)$}.
\end{equation*}
Thus, we get 
\begin{align*}
    Q_{U_h}(\xi_j)\le |h|\int_{B_{r_{j}}}\xi_{j}^2\,dx-\delta
    \int_{B_{r_j}}\frac{\xi_{j}^2}{|x|^2}\,dx<0
\end{align*}
for all $j$ sufficiently large. Since $\mathrm{supp}(\xi_j)\cap \mathrm{supp}(\xi_k)=\phi$ if $j\neq k$, we have $m(U_h)=\infty$ for all $h>-2(N-2)$.

\vspace{5pt}

\noindent
\textbf{Case 2.} \textit{The case of $N=10$.}
 
 We first point out that $2(N-2)=\frac{(N-2)^2}{4}$. Thus, thanks to Proposition \ref{improved hardy} and Proposition \ref{improvedhardy2}, we verify that 
 $U_h$ is stable if and only if $h\le H$. As a result, the bifurcation diagram is of Type II if and only if $h\le H$. Moreover, we have $m(U_h)=0$ if $h\le H$ and $m(U_h)\ge 1$ if $h>H$. 

Thus, it suffices to prove that $m(U_{h})<\infty$. We define $X\subset H^1_{0,\mathrm{rad}}(B_1)$ as a maximal space such that $Q_{U_h}(\xi)<0$ for all $\xi\in X\setminus \{0\}$ and we define $Y\subset H^1_{0,\mathrm{rad}}(B_1^2)$ as a maximal space such that $<\left(-\Delta-h\right)\xi,\xi>_{L^2(B_1^2)}<0$ for all $\xi \in Y\setminus \{0\}$, where $B_1^2$ is the 2-dimensional unit ball. We consider the linear map $\Lambda:X\to Y$ defined for $\xi\in X$ by $\Lambda \xi(|x|)=|x|^{\frac{2-N}{2}}\xi(|x|)$. Since the map is injective, it holds $\dim X\le \dim Y$. We consider
the eigenvalue problem
\begin{equation*}
    (-\Delta-h) u = \mu u \hspace{4mm}\text{in $B^2_{1}$}, \hspace{4mm} u\in H^1_{0,\rm{rad}}(B_1^{2}).
\end{equation*}
Then, it is well-known that there are countable
eigenvalues of this equation and each of them is discrete. Since $\dim Y$ is equal to the number of negative eigenvalues for the above equation (see \cite[Proposition 1.5.1]{Dup} for example), it holds $\dim Y<\infty$ and thus we get the result.
\end{proof}

\section{Comparison results}
\label{separate sec}
In this section, we prove comparison results. From now on, we deal with the case where the weight $a$ is more general. The comparison results play an important role in studying the stability of $U_{*}$. We first introduce the following 
%focus on the fact that $V_{h}$ is stable in $B_{r_h}$.by using a similar method used in \cite{Guini,Gui,Baeni,Mi2015},stability/instability of the radial singular solution $U_{*}$, we need to provide global upper and lower estimates of $V_{*}$ by $V_h$. Comparison results play an important role in providing the estimates. We first introduce the following proposition which leads to the upper estimate of $V_{*}$ by $V_h$. \textbf{This proposition is proved by a new method based on ideas introduced by} .

\begin{proposition}
\label{comparison3}
Assume that $N=10$ and $a(r)$ satisfies $(A)$ and 
\begin{equation}
\label{atarimae}
(a/a_{h})'\le 0,\hspace{4mm}0<r<1
\end{equation}
for some $h>0$. Then, we have
\begin{equation}
    v(r,\beta)+\log a(r) <v_{h}(r,\gamma) + \log a_{h}(r),\hspace{4mm} 0<r<r_h\notag
\end{equation}
for all $\gamma>\beta$, where $r_h=\min\{1,\sqrt{H/h}\}$.
\end{proposition}
We briefly explain the idea of the proof. We introduce the following comparison result provided in \cite{Gui,Guini}.
\begin{lemma}[see \cite{Gui,Guini}]
\label{guilem}
Let $R>0$ and let $k(|x|)\in C^0(B_R)$ be a nonnegative radial function. We assume that $w_1(|x|)\in C^2(B_R)$ and $w_2(|x|)\in C^2(B_R)$ are radial functions which satisfy
\begin{equation*}
-\Delta w_1\le k(x)w_1 \hspace{4mm}\text{in $B_R\cap \{w_1>0\}$}
\end{equation*}
and
\begin{equation*}
-\Delta w_2 \ge k(x)w_2 \hspace{4mm}\text{in $B_R$}.
\end{equation*}
Moreover, we assume that $w_1(0)>0$ and $w_2>0$ in $B_{R}$. Then, we have $w_1>0$
in $B_R$.
\end{lemma}
We define $k(r)=e^{v_h(r,\gamma)+\log a_h}$ and we denote by $w_1$ that in \eqref{kiee}.
Then, we wish to provide the function $w_2$ and apply Lemma \ref{guilem}. In \cite{Gui, Guini}, the authors focused 
on the equation for which the separation result holds (i.e., no two solutions intersect). As a result, they constructed $w_2$ as the difference between the two solutions. We note that this idea was 
applied by many authors (see e.g.\cite{Baeni, Bae, Mi2015}).

In this case, we point out that the singular solution $V_h$ plays an important role in providing $w_2$. Indeed, we verify that $V_{h}-V_{h}(r_{h})$ is a stable solution of \eqref{comp11}. The fact means that the bifurcation diagram of \eqref{comp11} is of Type II. As a consequence, we construct $w_2$ as the difference between two solutions for \eqref{comp11}. Moreover, we point out that we cannot apply Lemma \ref{kiee} directly since we do not know the sign of $-\Delta (\log a_h -\log a)$. By combining the idea of the proof of Lemma \ref{guilem} with Green's first identity, we get the result
from the assumption (\ref{atarimae}) only. 
\begin{proof}[Proof of Proposition \ref{comparison3}]
    Let $\gamma>\beta$ and we define
\begin{align}
\label{kiee}
w_1(r):=v_{h}(r,\gamma)+\log a_{h} -v(r,\beta)-\log a
\end{align}
and
\begin{align}
w_0(r):=v_h(r,\gamma)-v_h(r_h,\gamma).\notag
\end{align}
Then, we claim that there exist $\mu>e^{v_h(r_h,\gamma)}$ and a stable solution of
\begin{align}
\label{comp11}
  \left\{
\begin{alignedat}{4}
 -\Delta X&= \mu a_h(|x|)e^{X} &\hspace{2mm} &\text{in } B_{r_h},\\
\ X&> 0 & &\text{in $B_{r_h}$},\\
X&=0 &&\text{on $B_{r_h}$}
\end{alignedat}
\right.
\end{align}
such that $w_0<X$ in $B_{r_h}$. Indeed, we easily know that $w_0$ is a solution of \eqref{comp11} for $\mu=e^{v_h(r_h,\gamma)}$. Moreover, we verify that $V_h -V_h(r_h)$ is the radial singular solution of \eqref{comp11} for $\mu=e^{V_{h}(r_h)}$. Thanks to Proposition \ref{improved hardy}, the singular solution is stable. Thus, we verify that the bifurcation diagram of \eqref{comp11} is of Type II. It means that the bifurcation curve has only the stable branch and thus we get the claim.

We define 
\begin{equation*}
    w_2(r)=X(r)-w_0(r).
\end{equation*}
Then, it holds $w_2(r)>0$ for $0<r<r_h$.
Suppose the contrary, i.e., we assume that there exists $0<R<r_h$ such that $w_1(r)>0$ for $0<r<R$ and $w_1(R)=0$. Then, we have $w_1 '(R)\le 0$. By a direct calculation, it follows
\begin{align}
&-\Delta w_1 = e^{v_{h}(r,\gamma)+\log a_{h}}-e^{v(r,\beta)+\log a}-\Delta\left(\log a_h-\log a\right)\notag\\
&\hspace{18mm}\le e^{v_{h}(r,\gamma)+\log a_{h}}w_1-\Delta\left(\log a_h-\log a\right),\notag\\
&-\Delta w_2= e^{X+\log \left(\mu a_h\right)}-e^{v_h(r,\gamma)+\log a_h}\ge(\log\mu-v_h(r_h,\gamma)+w_2)e^{v_{h}(r,\gamma)+\log a_h}\notag
\end{align}
for all $0<r<R$. In particular, since $w_2$ is superharmonic, we have $w_2'\le 0$ for all $0<r<R$ (see Lemma \ref{rough lemma}).
Moreover, by the assumption \eqref{atarimae}, it holds $(\log a_h-\log a)'\ge 0$ in $B_1$.
Thus, it follows by Green's identity that
\begin{align}
    \omega_{N}R^N w_1 '(R) &w_2(R)
    = \int_{B_R}w_2\Delta w_1-w_1\Delta w_2\,dx\notag\\
    &=\int_{B_R}\log \left(\frac{\mu}{e^{v_h(r_h,\gamma)}}\right) e^{v_{h}(r,\gamma)+\log a_h}w_1+\Delta\left(\log a_h -\log a\right)w_2\,dx\notag\\
    &>\int_{\partial B_R} 
    \left(\nabla\left(\log a_h-\log a\right)\cdot\frac{x}{R}\right)w_2\,d\mathcal{H}^{N-1}\notag\\
    &\hspace{10mm}-\int_{B_R} \nabla \left(\log a_h-\log a\right)\cdot\nabla w_2\,dx\notag\\
    &\ge 0\notag,
\end{align}
which contradicts that $w_1'(R)\le 0$ and $w_2(R)>0$.
\end{proof}
Proposition \ref{comparison3} implies the separation property of $v$ in $B_{r_h}$. To be more precise, the following proposition holds.
\begin{proposition}
\label{comparison2}
Assume that $N=10$ and $a(r)$ satisfies $(A)$ and \eqref{atarimae} for some $h>0$.
Then, we have
\begin{equation}
    v(r,\beta)<v(r,\gamma),\hspace{4mm} 0<r<r_h\notag
\end{equation}
for all $\gamma>\beta$, where $r_h$ is that in Proposition \ref{comparison3}.
\end{proposition}
\begin{proof}
Let $\gamma>\beta$ and we define
\begin{align}
    &w_1(r):=v(r,\gamma)-v(r,\beta),\notag\\
    &w_0(r,\gamma):=v(r,\gamma)-v(r_h,\gamma).\notag
    \end{align}
Then, we claim that there exist $\mu>e^{v(r_h,\gamma)}$ and a stable solution of
\begin{align}
\label{comp12}
  \left\{
\begin{alignedat}{4}
 -\Delta X&= \mu a(|x|)e^{X} &\hspace{2mm} &\text{in } B_{r_h},\\
\ X&> 0 & &\text{in $B_{r_h}$},\\
X&=0 &&\text{on $B_{r_h}$}
\end{alignedat}
\right.
\end{align}
such that $w_0(r,\gamma)<X(r)$ in $B_{r_h}$. Indeed, we easily know that $w_0(r,\gamma)$ is a solution of \eqref{comp12} for $\mu=e^{v(r_h,\gamma)}$. Let $V_{*}$ be the limit function of $v(r,\gamma)$ as $\gamma\to\infty$. By Theorem \ref{singularth}, there exists the unique radial singular solution $(\mu_*,U_*)$ of \eqref{comp12} such that  $w_0(r,\gamma)\to U_{*}$ in $C^2_\mathrm{loc}((0,r_h])$. Thus, we have $\mu_*=e^{V(r_h)}$ and $U_{*}=V_{*}-V(r_h)$. Moreover, thanks to Proposition \ref{comparison3} and Theorem \ref{singularth}, it follows 
\begin{equation*}
    V_{*}+\log a\le V_h + \log a_h, \hspace{4mm} 0<r<r_h.
\end{equation*}
Therefore, thanks to Proposition \ref{improved hardy}, we verify that 
the bifurcation diagram of \eqref{comp12} is of Type II. It means the bifurcation curve has only the stable branch and thus we get the claim.

We define 
\begin{equation*}
    w_2(r)=X(r)-w_0(r,\gamma).
\end{equation*}
Then, it holds $w_2(r)>0$ for $0<r<r_h$.
Suppose the contrary, i.e., we assume that there exists $0<R<r_h$ such that $w_1(r)>0$ for $0<r<R$ and $w_1(R)=0$. Then, we have $w_1 '(R)\le 0$. By a direct calculation, it follows
\begin{align}
&-\Delta w_1 = e^{v(r,\gamma)+\log a}-e^{v(r,\beta)+\log a }\le e^{v(r,\gamma)+\log a}w_1,\notag\\
&-\Delta w_2= e^{X+\log \left(\mu a\right)}-e^{v(r,\gamma)+\log a}\ge(\log \mu-v(r_h,\gamma)+w_2)e^{v(r,\gamma)+\log a}\notag
\end{align}
for all $0<r<R$. Thus, it follows by Green's identity that
\begin{align}
    \omega_{N}R^N w_1 '(R) w_2(R)
    &= \int_{B_R}w_2\Delta w_1-w_1\Delta w_2\,dx\notag\\
    &=\int_{B_R}\log \left(\frac{\mu}{e^{v(r_h,\gamma)}}\right)e^{v(r,\gamma)+\log a}w_1\,dx\notag\\
    &>0\notag,
\end{align}
which contradicts that $w_1'(R)\le 0$ and $w_2(R)>0$.
\end{proof}
As a consequence of Proposition \ref{comparison3} and Proposition \ref{comparison2}, we get the following
\begin{corollary}
\label{cor}
    Assume that $N=10$ and $a(r)$ satisfies $(A)$ and the assumption of $(i)$ in Theorem \ref{Mainthm}. Then, we have 
    \begin{equation*}
    v(r,\beta)+\log a(r) <v_{H}(r,\gamma) + \log a_{H}(r)\hspace{4mm}\text{and}\hspace{4mm}
    v(r,\beta)<v(r,\gamma),
\end{equation*}
for all $\gamma>\beta$ and $0<r<1$.
\end{corollary}
Proposition \ref{comparison3} and Corollary \ref{cor} imply that $V_{*}+\log a $ stays below $V_{h}+\log a_{h}$ in $B_{r}$
as long as $V_{h}$ is stable in $B_{r}$ if the assumption \eqref{atarimae} holds. In particular, if \eqref{atarimae} holds for some $h\le H$, the function $V_{*}+\log a $ stays below $V_{h}+\log a_{h}$ in $B_{1}$. This fact is useful to study the stability of $U_{*}$. On the other hand, if \eqref{atarimae} holds for some $h>H$, we get the following proposition by focusing on the stable branch.
\begin{proposition}
\label{comparison1}
Assume that $N=10$ and $a(r)$ satisfies $(A)$ and the assumption of $(iii)$ in Theorem \ref{Mainthm}. Then, there exists a nondecreasing function $\varepsilon:[0,1]\to [0,1]$ 
such that  
\begin{equation}
\label{seisitu epsilon}
    \varepsilon(|x|)\in C^2(\overline{B_1}),\hspace{4mm} \varepsilon(1/2)>0,
\end{equation}
and
\begin{equation}
      v(r,\gamma)+\log a >v_0(r,\beta)+\log \left(1+\frac{H+\varepsilon}{2(N-2)}r^2\right),\hspace{4mm} 0<r<1\notag
\end{equation}
for any $\lVert u^{*}\rVert_{L^{\infty}(B_1)}+\log \lambda^*>\gamma>\beta>0$, where $\varepsilon$ is depending only on $a$ and $N$.

\end{proposition}
\begin{proof}
Thanks to the assumption of (iii) in Theorem \ref{Mainthm}, we can take a nondecreasing convex function $\varepsilon:[0,1]\to [0,1]$ such that
\begin{equation}
\varepsilon(|x|)\in C^2(\overline{B_1}),\hspace{2mm}
\varepsilon>0\hspace{2mm}\text{in $B_1\setminus B_{\frac{1}{4}}$},\hspace{2mm}\text{and} \hspace{2mm}(\log a-\log a_{H+\varepsilon})'\ge 0\hspace{4mm} \text{in $B_1$},\notag
\end{equation}
where
\begin{equation*}
    a_{H+\varepsilon}(r):=\left(1+\frac{H+\varepsilon(r)}{2(N-2)}r^2\right)e^{\frac{H+\varepsilon(r)}{2N}r^2}.
\end{equation*}
%We take $m\in \mathbb{N}$ such that $(\log a_H-\log a)^{(m-1)}(0)=0$ and $(\log a_H-\log a)^{(m)}(0)\neq0$. 
%Since $a(0)=1$, $a'(0)=0$, and $(a_H/a)'<0$ in $(0,1]$, there exists $\varepsilon>0$ such that $a_{H+\varepsilon r^{m-2}}/a$ is nonincreasing.
Let $u^* (0)+\log \lambda^{*}>\alpha>\gamma>\beta$ and we define
\begin{align}
    &w_1(r):=v(r,\gamma)+\log a -v_0(r,\beta)-\log\left(1+\frac {H+\varepsilon}{2(N-2)}r^2\right),\notag\\
    &w_2(r):= v(r,\alpha)-v(r,\gamma)-\log \frac{\lambda(\alpha)}{\lambda(\gamma)}.\notag
\end{align}
We claim that $w_2(r)>0$ for $0<r<1$. Indeed, we set $u_1:=v(r,\alpha)-\log \lambda(\alpha)$ and $u_{2}:=v(r,\gamma)-\log \lambda(\gamma)$. Then, $(\lambda(\alpha),u_1)$ and $(\lambda(\gamma), u_2)$ are minimal solutions of \eqref{gelfand} and it holds $\lambda(\alpha)>\lambda(\gamma)$. Thus, it holds $u_1<u_2$ and we get the claim. Suppose the contrary, i.e., we assume that that there exists $0<R<1$ such that $w_1(r)>0$ for $0<r<R$ and $w_1(R)=0$. Then, we have $w_1 '(R)\le 0$. On the other hand, we verify that $v_0(r,\beta)$ is increasing for $\beta$ by Corollary \ref{cor}. Since $v_0(r,\beta)\to -2\log r+\log 2(N-2)$, it holds $e^{v_0 (r,\beta)}<2(N-2)r^{-2}$. Thus, we have for $0<r<R$,
\begin{align}
-&\Delta w_1(r) = e^{v(r,\gamma)+\log a}-e^{v_0(r,\beta)+\log \left(1+\frac{H+\varepsilon}{2(N-2)}r^2\right)}+\frac{H+\varepsilon}{2(N-2)}r^2e^{v_0(r,\beta)}\notag\\
&\hspace{40mm}-\Delta\left(\log a-\log(1+\frac{H+\varepsilon}{2(N-2)}r^2)\right)\notag\\
&\le e^{v(r,\gamma)+\log a}w_1+H+\varepsilon-\Delta\left(\log a-\log(1+\frac{H+\varepsilon}{2(N-2)}r^2)\right)\notag\\
&\le e^{v(r,\gamma)+\log a}w_1-\Delta\left(\log a-\log a_{H+\varepsilon}\right)\hspace{4mm}\text{(since $\varepsilon'\ge 0, \varepsilon''\ge 0$) }.\notag
\end{align}

%Here, we used the fact that $\Delta \log(1+pr^2)< \Delta \log (1+qr^2)$ for $0<p<q$.
Moreover, we have for $0<r<1$,
\begin{align}
-\Delta w_2= e^{v(r,\alpha)+\log a}-e^{v(r,\gamma)+\log a}> \left(\log \frac{\lambda(\alpha)}{\lambda(\gamma)}+w_2\right)e^{v(r,\gamma)+\log a}\notag.
\end{align}
In particular, since $w_2$ is superharmonic, we have  $w_2'\le 0$ for all $0<r<1$ (see Lemma \ref{rough lemma}). Moreover, we recall that $(\log a-\log a_{H+\varepsilon})'\ge 0$. Thus, it follows by Green's identity that
\begin{align}
    \omega_{N}R^N w'_1 (R) &w_2(R)
    = \int_{B_R}w_2\Delta w_1-w_1\Delta w_2\,dx\notag\\
    &>\int_{B_R}\Delta\left(\log a -\log a_{H+\varepsilon}\right)w_2\,dx\notag\\
    &\ge \int_{\partial B_R} 
    \left(\nabla\left(\log a-\log a_{H+\varepsilon}\right)\cdot\frac{x}{R}\right)w_2\,d\mathcal{H}^{N-1}\notag\\
    &\hspace{10mm}-\int_{B_R} \nabla \left(\log a-\log a_{H+\varepsilon}\right)\cdot\nabla w_2\,dx\notag\\
    %&\ge \int_{\partial B_R} 
    %\left(\nabla\left(\log a_{H+2\varepsilon}(\lambda_{\gamma}^{-1/2} r)-\log a_{H+\varepsilon}(\lambda_{\gamma}^{-1/2} r)\right)\cdot\frac{x}{R}\right)w_2\,d\mathcal{H}^{N-1}\notag\\
    &\ge 0\notag,
\end{align}
which contradicts that $w_1'(R)\le 0$ and $w_2(R)>0$.
\end{proof}

%%%%%%%%%%%%%%%%%%%%%%%%%%%%%%%%%%%%%%%%%%%%%%%%%%%%%%%%%%%%%%%%%%%%%%%%%%%%%%%%%%%%%%%%%%%%%%%%%%%%%%%%%%%%%%%%%%%%%%%%%%%%%%%%%%%%%%%%%%%%%%%%%%%%%%%%%%%%%%%%%%%%%%%%%%%%%%%%%%%%%%%%%%%%%
\section{Classification of bifurcation diagrams.}
\label{mainsec}
In this section, we classify the bifurcation diagrams. For simplicity, we set $v'(r,\beta):=\frac{d}{d\beta} v(r,\beta)$ and $v''(r,\beta)=\frac{d^2}{d\beta^2}v(r,\beta)$. Moreover, we set $\lambda'(\beta):=\frac{d}{d\beta}\lambda(\beta)$ and $\lambda''(\beta):=\frac{d^2}{d\beta^2}\lambda(\beta)$.

\subsection{When \texorpdfstring{$N\le 9$}{Lg}}
In this subsection, following the idea of \cite{Mi2014},
we prove Theorem \ref{39th}.
We define
\begin{align*}
\hat{v}_{\beta}(r):=v(e^{-\frac{\beta}{2}}r,\beta)-\beta\hspace{4mm}\text{and}\hspace{4mm}
\hat{V}_{\beta}(r,\beta)=V_{*}(e^{-\frac{\beta}{2}}r)-\beta.
\end{align*}
Thus, $\hat{v}_{\beta}$ satisfies
\begin{equation}
\left\{
\begin{alignedat}{4}
 &-\Delta \hat{v}_{\beta}= \hat{a}_{\beta}(r)e^{\hat{v}_{\beta}}, \hspace{14mm}0<r<\infty,\\
&\hat{v}_{\beta}(0)=0, \hspace{4mm} \frac{d}{dr} \hat{v}_{\beta}(0)=0,\notag
\end{alignedat}
\right.
\end{equation}
where $\hat{a}_{\beta}(r)=a(e^{-\frac{\beta}{2}}r)$.
We prove the following 
\begin{proposition}
\label{inter}
Let $N\le 9$. We assume that $a(r)$ satisfies
$(A)$. Then
\begin{equation*}
    \hat{v}_{\beta}\to v_{0}(r,0) \hspace{2mm}\text{in $C^1_{\mathrm{loc}}[0,\infty)$}
\end{equation*}
and
\begin{equation*}
    \hat{V}_{\beta}\to -2\log r + \log 2(N-2)
    \hspace{2mm}\text{in $C^0_{\mathrm{loc}}(0,\infty)$}.
\end{equation*}
\end{proposition}
\begin{proof}
    By Lemma \ref{rough lemma}, it holds $\frac{d}{dr} \hat{v}_{\beta}\le 0$ for all $0<r<\infty$. Moreover, for each bounded interval $I$, there exists $\beta_0$ such that $1\le\hat{a}_{\beta}\le 2$ in $I$ if $\beta>\beta_0$. Since
    \begin{equation*}
           \frac{d}{dr}(r^{N-1}\frac{d}{dr}\hat{v}_{\beta})= -\hat{a}_{\beta}(r)e^{\hat{v}_{\beta}}r^{N-1}\ge -2r^{N-1}\hspace{4mm}\text{in $I$}, 
    \end{equation*}
we have
    \begin{equation*}
        \frac{d}{dr}\hat{v}_{\beta}(r)\ge -2r^{1-N}\int^{r}_{0}s^{N-1}\,ds=-\frac{2}{N}r \hspace{4mm}\text{in $I$}
    \end{equation*}
    and thus it holds $-r^2/N\le\hat{v}_{\beta}\le 0$ in $I$. 
     By Ascoli-Arzel\'a theorem, there exist
     $\beta_m$ and $\hat{v}_{\infty} \in C^{0}_{\mathrm{loc}}[0,\infty)$ such that $\hat{v}_{\beta_m}\to \hat{v}_{\infty}$ in $ C^{0}_{\mathrm{loc}}[0,\infty)$ as $m\to\infty$. Moreover, since 
     \begin{equation*}
      -\frac{d}{dr} \hat{v}_{\beta_m} (r)=  r^{1-N}\int^{r}_{0}s^{N-1} \hat{a}_{\beta_m}(s) e^{\hat{v}_{\beta_m}}\,ds
     \end{equation*}
     and $\hat{a}_{\beta_m}(r)\to 1$ in $C^{0}_{\mathrm{loc}}[0,\infty)$ 
     as $m\to \infty$, $\frac{d}{dr}v_{\beta_m}$ converges in $C^{0}_{\mathrm{loc}}[0,\infty)$. Since, $\frac{d}{dr}$ is a closed operator with a domain $C^{1}(I)$ in $C^{0}(I)$ for each bounded interval $I$, we get $\hat{v}_{\infty}\in C^{1}_{\mathrm{loc}}[0,\infty)$ and $\hat{v}_{\beta_m}\to \hat{v}_{\infty}$ in $C^{1}_{\mathrm{loc}}[0,\infty)$
     as $m\to\infty$. In particular, 
     it holds
     \begin{equation*}
         \hat{v}_{\infty}(r)=-\int^{r}_{0}\,ds\int^{s}_{0}\left(\frac{t}{s}\right)^{N-1}e^{\hat{v}_{\infty}}\,dt.
     \end{equation*}
     Thus, by the uniqueness of the solution of the ODE, we get $\hat{v}_{\infty}(r)=v_0(r,0)$. It implies $\hat{v}_{\beta}\to v_0(r,0)$ in $C^{1}_{\mathrm{loc}}[0,\infty)$  as $\beta\to\infty$.

By Proposition \ref{asymptotic prop}, it holds $V_{*}(r)+2\log r\to \log 2(N-2)$ as $r\to 0$ and thus we have $\hat{V}_{\beta}(r)+2\log r \to  \log 2(N-2)$ as $e^{-\frac{\beta}{2}}r\to 0$. Since, for each bounded interval $I$, it holds $e^{-\frac{\beta}{2}}r\to 0$ as $\beta\to\infty$, we get the result.
\end{proof}
We denote the zero number of the function $u(r)$ on the interval $I$ by 
\begin{equation*}
    \mathcal{Z}_{I}[u(r)]=\sharp\{r\in I: u(r)=0\}.
\end{equation*}
Then it is well known that 
\begin{equation*}
    \mathcal{Z}_{[0,\infty)}[v_0(r,0)-2\log r + \log 2(N-2)]=\infty \hspace{4mm}\text{if $N\le 9$}.
\end{equation*}
By using Proposition \ref{inter}, we get the following
\begin{proposition}
\label{interprop}
  Under the same assumption of Proposition \ref{inter}, we have
  \begin{equation*}
      \mathcal{Z}_{[0,1]}[v(r,\beta)-V_{*}(r)]\to\infty\hspace{4mm}\text{as $\beta\to \infty$}.
    \end{equation*}
\end{proposition}

\begin{proof}[Proof of Theorem \ref{39th}.]
We remark that each zero of $v(r,\beta)-V_{*}(r)$ is simple and discrete by the uniqueness of the solution of the ODE.
Moreover, since $v(0,\beta)-V_{*}(r)=-\infty$, we have $\mathcal{Z}_{[0,2]}[v(r,\beta)-V_{*}(r)]<\infty$ for all $0\le\beta<\infty$. We set $h(r,\beta)=v(r,\beta)-V_{*}(r)$. In addition, let $r_0\in (0,2]$ and $\beta_0>0$ be constants such that $h(r_0,\beta_0)=0$. Since $\frac{d}{dr} h(r_0,\beta_0)\neq 0$, it follows by the implicit function theorem that there exists a $C^1$-function $r(\beta)$ such that $h(r(\beta),\beta)=0$ for $\beta$ near $\beta_0$ and $r(\beta_0)=r_0$. Let $M\in \mathbb{N}$ and 
we define 
\begin{equation*}
\beta_{M}:=\sup \{ \beta\ge 0 :\mathcal{Z}_{[0,1]}[h(r,\beta)]<M\}.
\end{equation*}
Thanks to Proposition \ref{interprop} and the above discussion, there exists $M_0\in \mathbb{N}$ such that we can define $\beta_M$ for all $M>M_0$ and it holds $0<\beta_{M}<\infty$ for all $M>M_0$. We fix $\beta_M$ and we denote by $r_{i}(\beta)$ the $i$-th zero of $h(\cdot,\beta)$. Then, we verify that $r_{M}(\beta_M)=1$ and $r_{M}(\beta)>1$ if $\beta<\beta_M$ and $\beta$ is close to $\beta_M$. Moreover, we know that $r_{M}(\beta)\le 1$ for all $\beta>\beta_M$ and $r_{M}(\beta)<1$ for a sufficiently large $\beta$. 
%Since $h(1,\beta)$ is analytic and $r_{M+1}(\beta)$ is continuous, there exists $\beta_{M}<\hat{\beta}<\beta_{M+1}$ such that $r_{M}(\beta)<1$ and $r_{M+1}(\beta)>1$ for all $\beta_{M}<\beta<\hat{\beta}$.

We define 
\begin{equation*}
    \beta_{\hat{M}}:=\sup\{\beta>\beta_{M}: r_{M}(\beta)=1\hspace{2mm}
\text{on}\hspace{2mm} [\beta_{M},\beta]\}. \end{equation*}
Then, we have $r_{M}(\beta_{\hat{M}})=1$ and $\beta_{\hat{M}}<\beta_{M+1}$.  Moreover, since $h(1,\beta)$ is analytic and $r_{M+1}(\beta)$ is continuous, there exists $\beta_{\hat{M}}<\hat{\beta}<\beta_{M+1}$ such that $r_{M}(\beta)<1$ and $r_{M+1}(\beta)>1$ for all $\beta_{\hat{M}}<\beta<\hat{\beta}$.
%Thus, we can take a $\beta$ close to $\beta_{\hat{M}}$ such that 
%$r_{M}(\beta)<1$ and $r_{M+1}(\beta)>1$. 
By the above facts, it follows that the sign of $h(1,\beta)=\log \lambda(\beta)-\log \lambda_{*}$ changes if $\beta$ crosses $[\beta_{M},\beta_{\hat{M}}]$. Thanks to Proposition \ref{interprop},
we know that $\lambda(\beta)$ oscillates around $\lambda_{*}$ infinity many times as $\beta\to\infty$. Therefore, we verify that the bifurcation is of Type I.
%Since $\alpha(\beta)\to \infty$ as $\beta\to\infty$, we verify that the bifurcation is of Type I.
\end{proof}

\subsection{When \texorpdfstring{$N=10$}{Lg}}
In this subsection, we prove Theorem \ref{Mainthm}. We begin by introducing the following proposition, which plays a key role in proving that the bifurcation diagram is of Type III. 

%\begin{lemma}
%Assume that $a(r)$ satisfies (A), %\eqref{atarimae}, and the assumption of $(iii)$ in Theorem \ref{Mainthm}. Then, we have
    %\begin{equation*}
     %   U_{H}+\log \lambda^{*}_{H} \le V+\log a \le U_{h}+\log \lambda^{*}_{h},\hspace{4mm}0<r<r_h
   % \end{equation*}
    %and
    %\begin{equation*}
    %\frac{d}{dr}\left(U_{h}+\log \lambda^{*}_{h}\right)\le\frac{d}{dr}V\le \frac{d}{dr}\left(U_{H}+\log \lambda^{*}_{H}\right),\hspace{4mm}0<r<r_h
    %\end{equation*}
    %In particular, 
%\end{lemma}
\begin{proposition}
\label{mainprop}
Let $N= 10$ and we assume that $a$ satisfies \eqref{atarimae} for some $h>0$. Then there exists $\beta_0>0$ such that if $(\lambda(\beta), u(r,\alpha(\beta)))$ is a radial solution of \eqref{gelfand} satisfying $\beta>\beta_{0}$, then $\lambda'(\beta)\neq 0$.    
\end{proposition}
%We state the idea of the proof. We take a sequence $\beta_k>0$ such that $\beta_{k}\to \infty$ as $k\to\infty$ and $\lambda'(\beta_k)=0$ by assuming to the contrary. Then, we know that $0$ is an eigenvalue of $-\Delta-ae^v$ and $v'(r,\beta_k)$ is an eigenfunction with respect to $0$. From this fact, we can expect that $0$ is an eigenvalue of $-\Delta-ae^{V_{*}}$. By using the specific change of variables \eqref{spe} and a convergence argument,
%we justify it. If $-\Delta-ae^{V_{*}}$ does not have the eigenvalue of $0$, it is a contradiction. Otherwise, we focus on $v''(r,\beta_k)$ and prove that $(-1)^{i-1}\lambda''(\beta_k)>0$ if $k$ is sufficiently large, where $i\in \mathbb{N}$ is independent of $k$, which is a contradiction.
\begin{proof}
Let $i$ be a natural number depending only on $a$ to be defined later. Then, it suffices to prove the following: there exists $\gamma>0$ such that if it follows $\gamma<\beta$ and $\lambda'(\beta)=0$, then $(-1)^{i-1}\lambda''(\beta)>0$. Suppose the contrary, i.e., we assume that there exist a sequence of radial solutions of \eqref{gelfand} $(\lambda(\beta_k), u(r,\alpha(\beta_k)))$ such that  $\lambda'(\beta_k)=0$, $(-1)^{i-1}\lambda''(\beta_k)\le 0$, and
$\beta_k\to \infty$ as $k\to\infty$.
We set $v_{k}(r):= u(r,\alpha(\beta_k))+\log \lambda(\beta_k)$. We recall that $v_k$ is the solution of \eqref{odev} for $\beta=\beta_k$. %Moreover, thanks to Theorem \ref{singularth}, it holds  $\beta_k\to\infty$. 

By a direct calculation, it follows that $v_{k}'$ and $v_{k}''$ satisfy
\begin{equation}
\left\{
\begin{alignedat}{4}
 &-\Delta v_{k}'= a(r)e^{v_{k}}v_{k}', \hspace{14mm}0<r\le 1,\\
&v_{k}'(0)=1, \hspace{1mm} \frac{d}{dr} v_{k}'(0)=0, \hspace{1mm} v_{k}'(1)=0\notag
\end{alignedat}
\right.
\end{equation}
and
\begin{equation}
\left\{
\begin{alignedat}{4}
-\Delta v_{k}''&= a(|x|)e^{v_{k}}({v_{k}'}^2+v_{k}''), \hspace{4mm}&&\text{in $B_1$},\\
v_{k}''&=\frac{\lambda''(\beta_k)}{\lambda(\beta_k)} &&\text{on $\partial B_1$}.\notag
\end{alignedat}
\right.
\end{equation}
We remark that by the assumption \eqref{atarimae} and Proposition \ref{comparison2},
we have $v'_{k}\ge 0$ for $0<r<r_{h}$, where $r_h$ is that in Proposition \ref{comparison3}.
In addition, we remark that 
$v_{k}\to V_{*}$ in $C^2_{\mathrm{loc}}(0,1]$ and $\lambda_{*}:=\lim_{k\to \infty} \lambda(\beta_k)>0$, where $V_{*}$ is that in \eqref{singulareq}. 

We define 
\begin{equation}
\label{spe}
    w_k=\frac{|x|^{\frac{N-2}{2}} v_{k}'}{\lVert |x|^{\frac{N-2}{2}} v_{k}' \rVert_{L^2(B^{2}_{1})}},
\end{equation}
where $B^{2}_{1}$ is the 2-dimensional unit ball. Then, $w_k$ satisfies 
\begin{equation}
\label{odew}
\left\{
\begin{alignedat}{4}
 -\Delta w_{k}&= \left(a(|x|)e^{v_{k}}-\frac{(N-2)^2}{4|x|^2}\right) w_{k}, \hspace{4mm}&&\text{in $B^{2}_1$},\\
w_{k}&=0 &&\text{on $\partial B^{2}_1$}
\end{alignedat}
\right.
\end{equation}
and it satisfies $w_{k}(r)\ge 0$ in $0<r<r_h$. Thanks to Theorem \ref{singularth} and Proposition \ref{comparison3}, it holds
\begin{equation*}
    a(|x|) e^{V_{*}} \le e^{V_{H}+\log a_{H}}=2(N-2)|x|^{-2}+ H \hspace{4mm}\text{in $B_{r_h}$}.
\end{equation*}
As a result, thanks to Proposition \ref{comparison2} and the fact that $v_{k}\to V_{*}$ in $C^2_{\mathrm{loc}}(0,1]$, we get 
\begin{equation*}
  a(|x|)e^{v_{k}}-2(N-2)|x|^{-2}\le C
\end{equation*}
for some $C$ depending only on $a$. We remark that $2(N-2)=\frac{(N-2)^2}{4}$ when $N=10$.
Thus, it follows the energy estimate
\begin{equation*}
    \int_{B^{2}_{1}} |\nabla w_k|^2\,dx=
    \int_{B^{2}_{1}}\left(a(|x|)e^{v_{k}}-\frac{(N-2)^2}{4|x|^2}\right) w^2_{k}\,dx\le C\int_{B^{2}_1} w^2_{k}\,dx=C.
\end{equation*}
Therefore, there exists $w\in H^{1}_{0}(B^2_{1})$ such that $w_k\rightharpoonup w$ in $H^{1}(B^{2}_1)$ by taking a subsequence if necessary. In particular, we get $w_k\to w $ in $L^q(B^{2}_1)$ for any $1\le q<\infty$. Thus $\lVert w\rVert_{L^2(B^2_{1})}=1$ and 
$w$ satisfies
\begin{equation}
\left\{
\begin{alignedat}{4}
 -\Delta w&= \left(a(|x|)e^{V_{*}}-\frac{(N-2)^2}{4|x|^2}\right) w, \hspace{4mm}&&\text{in $B^2_{1}$},\\
w&=0 &&\text{on $\partial B^2_{1}$}.\notag
\end{alignedat}
\right.
\end{equation}
Moreover, $w$ is radially symmetric. Since $V_{*}$ satisfies \eqref{singulareq}, it holds
\begin{equation*}
    K(|x|):=\left(a(|x|)e^{V_{*}}-\frac{(N-2)^2}{4|x|^2}\right)\in C^{0}(B^{2}_1).
\end{equation*} 
Thus, by the elliptic theory \cite{gil} and Lemma \ref{apenlem2}, we verify that there are countable eigenvalues of $-\Delta-K$ and each eigenfunction is in $C^2(\overline{B^{2}_1})$. If $0$ is not an eigenvalue, it holds $w=0$, which contradicts $\lVert w\rVert_{L^2(B^2_{1})}=1$. 

From now on, we deal with the remaining case where $0$ is an eigenvalue. In this case, we can define $i$ as the number such that $(-1)^{i-1}\frac{d}{dr}w(1)<0$ since it holds $\frac{d}{dr}w(1)\neq 0$. We remark that $i$ is depending only on $a$. Moreover, we define

\begin{equation*}
\mu_k:=  \frac{\lambda''(\beta_k)}{\lambda(\beta_k)}\hspace{4mm}\text{and}\hspace{4mm}z_{k}=r^{\frac{N-2}{2}} v_{k}''- \mu_{k} r^2.
\end{equation*}
Then, it holds $(-1)^{i-1}\mu_k\le 0$. Moreover, $z_k\in C^2_{0}(\overline{B^2_{1}})$ satisfies
\begin{equation}
-\Delta z_{k}= \left(a(|x|)e^{v_{k}}-\frac{(N-2)^2}{4|x|^2}\right)(z_{k}+\mu_{k}|x|^2)+a(|x|)|x|^
{\frac{N-2}{2}}e^{v_{k}}{v_{k}'}^2+4\mu_k\hspace{2mm}\text{in $B_1^{2}$}\notag.
\end{equation}

%\begin{equation}
%\left\{
%\begin{alignedat}{4}
% -\Delta z_{k}&= \left(a(|x|)e^{v_{k}}-\frac{(N-2)^2}{4|x|^2}\right)(z_{k}+\mu_{k}|x|^2)+a(|x|)|x|^
%{\frac{N-2}{2}}e^{v_{k}}{v_{k}'}^2+4\mu_k\notag\\
%&\hspace{100mm}\text{in $B^2_{1}$}\notag\\
%z_{k}&=0\hspace{4mm}\text{on $\partial %B_1^{2}$}.\notag
%\end{alignedat}
%\right.
%\end{equation}
Thanks to the above equation and \eqref{odew}, we get 
\begin{align*}
    0&=\int_{B_{1}^2} \left(-\Delta w_k - \left(a(|x|)e^{v_{k}}-\frac{(N-2)^2}{4|x|^2}\right)w_k\right)z_k\,dx\\
    &=\int_{B_{1}^2} \left(-\Delta z_k - \left(a(|x|)e^{v_{k}}-\frac{(N-2)^2}{4|x|^2}\right)z_k\right)w_k\,dx= \mathcal{A}_k+ \mu_k \mathcal{B}_k,
\end{align*}
where
\begin{align*}
    &\mathcal{A}_k:= \int_{B^2_{1}}a(|x|)|x|^{\frac{N-2}{2}}e^{v_{k}}{v_{k}'}^2 w_k\,dx,\\
&\mathcal{B}_k:=\int_{B^2_{1}}\left(\left(a(|x|)e^{v_{k}}-\frac{(N-2)^2}{4|x|^2}\right)|x|^2+\Delta |x|^2\right)w_k\,dx.
\end{align*}
We claim that there exists $k_0$ such that 
$A_k>0$, $(-1)^{i-1}B_k<0$ for all $k>K_0$, which contradicts the above equality.
Indeed, we recall that $w_k\to w $ in $L^q(B^{2}_1)$ for any $1\le q<\infty$ and $v_k\to V_{*}$ in $C^2_{\mathrm{loc}}(0,1]$. Thanks to the above notations and Fatou's Lemma, it holds
\begin{align*}
    \liminf_{k\to\infty}\frac{\mathcal{A}_k}{\lVert |x|^{\frac{N-2}{2}}v_k'\rVert_{L^2(B^2_{1})}^3}&=\liminf_{k\to\infty}\int_{B^2_{1}}a(|x|)|x|^{\frac{2-N}{2}}e^{v_{k}}w^3_{k}\,dx\\
    &\ge \liminf_{k\to\infty} \int_{B^2_{r_0}}a(|x|)|x|^{\frac{2-N}{2}}e^{v_{k}}w^3_{k}\,dx-C\\
    &=\infty
\end{align*}
and
\begin{align*}
\lim_{k\to\infty}\mathcal{B}_k&=\int_{B^2_{1}} \left(K(|x|)|x|^2+\Delta |x|^2\right)w\,dx\\
&=\int_{B^2_{1}} K(|x|)w\,dx+\int_{B^2_{1}} \left(K(|x|)(|x|^2-1)+\Delta (|x|^2-1)\right)w\,dx\\
&=\int_{B^2_{1}} -\Delta w\,dx=2\pi \frac{d}{dr}w(1).
\end{align*}
Thus, we get the claim.  
\end{proof}

%textbf{Then, we prove the following}
\begin{proof}[Proof of Theorem \ref{Mainthm}.]
We first assume that $a(r)$ satisfies $(A)$ and the assumption of $(i)$. In this case, thanks to Theorem \ref{singularth} and Corollary \ref{cor}, it holds
\begin{equation*}
    V_{*}+ \log a \le V_{H} + \log a_{H},
\end{equation*}
 where $a_H$ is that in Theorem \ref{singularthm},
$V_{*}$ is the solution of \eqref{singulareq}, and $V_H$ is the solution of \eqref{singulareq} for $a=a_H$.
We recall $2(N-2)=\frac{(N-2)^2}{4}$ and $a_{H}e^{V_H}=2(N-2)|x|^{-2}+H$. Thus, for any $\xi\in H^1_{0}(B_1)$, we have
\begin{align}
  Q_{U_{*}}(\xi)&=\int_{B_1} \left(|\nabla \xi|^2-\lambda_{*}ae^{U_{*}}\xi^2\right)\,dx
    =\int_{B_1}\left(|\nabla \xi|^2-ae^{V_{*}}\xi^2\right)\,dx\notag\\
    &\ge \int_{B_1}|\nabla \xi|^2\,dx- 2(N-2)\int_{B_1} \frac{\xi^2}{|x|^2}\,dx- H\int_{B_1} \xi^2\,dx\notag\\
    &=\int_{B_1}|\nabla \xi|^2\,dx- \frac{(N-2)^2}{4}\int_{B_1} \frac{\xi^2}{|x|^2}\,dx- H\int_{B_1} \xi^2\,dx.\notag
    \end{align}
By Proposition \ref{improved hardy}, $U_{*}$ is stable. Thanks to Proposition \ref{brezis V} and Theorem \ref{singularth}, we verify that the bifurcation diagram is of Type II.

Then we assume that $a(r)$ satisfies $(A)$ and the assumption of $(ii)$. We divide the proof into the following three steps.
\vspace{5pt}

\noindent
\textbf{Step 1.} \textit{$u^*$ is bounded.}

We recall the following notation: $u^*$ is the extremal solution for $\lambda=\lambda^*$. Suppose the contrary, i.e., we assume that $\lVert u^{*}\rVert_{ L^{\infty}(B_1)}=\infty$. Since $\lambda^*>0$, it holds $\lVert u^{*}\rVert_{ L^{\infty}(B_1)}+\log \lambda^*=\infty$. Then, thanks to Theorem \ref{singularth} and Proposition \ref{comparison1}, there exists a nondecreasing function $\varepsilon:[0,1]\to [0,1]$ which satisfies \eqref{seisitu epsilon} such that 
\begin{equation*}
    V_{*}+\log a\ge -2\log |x|+\log 2(N-2)+\log \left(1+\frac{H+\varepsilon(|x|)}{2(N-2)}|x|^2\right). 
\end{equation*}
Thus, for any $\xi\in H^1_{0}(B_1)$, we have
\begin{align}
  Q_{U_{*}}(\xi)&=\int_{B_1}\left(|\nabla \xi|^2-ae^{V_{*}}\xi^2\right)\,dx\notag\\
    &\le\int_{B_1}|\nabla \xi|^2\,dx- \frac{(N-2)^2}{4}\int_{B_1} \frac{\xi^2}{|x|^2}\,dx- \int_{B_1}(H+\varepsilon(|x|)) \xi^2\,dx.\notag
    \end{align}
By Proposition \ref{improvedhardy2}, we know that $U_{*}$ is unstable. Thus, it follows by Proposition \ref{brezis V} and Theorem \ref{singularth} that $u^* \in L^{\infty}(B_1)$.

\vspace{5pt}

\noindent
\textbf{Step 2.} \textit{The bifurcation curve turns at $\lambda=\lambda^*$}.

We set $\beta_0 := \lVert u^{*} \rVert_{L^\infty(B_1)}+\log \lambda^*$, We recall that equation \eqref{gelfand} has no weak solution for $\lambda>\lambda^{*}=\lambda(\beta_0)$. By the implicit function theorem, it holds $\lambda'(\beta_0)=0$. We recall that $\alpha(\beta)=\beta-\log \lambda(\beta)$. Thus it holds $\alpha'(\beta_0)=1$. Therefore, it suffices to prove $\lambda''(\beta_0)<0$.

We denote $v=v(r,\beta_0)$ and $\mu=\lambda''(\beta_0)/\lambda(\beta_0)$. Then, by a direct calculation, it follows that $v'$ and $v''$ satisfy
\begin{equation}
\label{eqv'}
\left\{
\begin{alignedat}{4}
 &-\Delta v'= a(r)e^{v}v', \hspace{14mm}0<r\le 1,\notag\\
&v'(0)=1, \hspace{1mm} \frac{d}{dr} v'(0)=0, \hspace{1mm} v'(1)=0
\end{alignedat}
\right.
\end{equation}
and
\begin{equation}
\left\{
\begin{alignedat}{4}
 -\Delta v'' &= ae^{v}({v'}^2+v'') \hspace{4mm}&&\text{in $B_1$},\notag\\
v''&=\mu &&\text{on $\partial B_1$}.\notag
\end{alignedat}
\right.
\end{equation}
Thus, by the Green's inequality, it holds

\begin{align*}
    0&=\int_{B_{1}} (-\Delta v' - a(|x|)e^{v}v')v''\,dx\\
    &=\int_{B_1} (-\Delta v'' -a(|x|)e^{v}v'')v'\,dx-
     \mu\int_{\partial B_1} v'\,d\mathcal{H}^{N-1}\notag\\
  &=\int_{B_1} a(|x|)e^v {v'}^3\,dx-
    \mu Nw_{N}\frac{d}{dr}v'(1)\notag,
\end{align*}
where $w_N$ is the measure of the $N$-dimensional unit ball. We claim that $v'>0$ in $B_1$ and $\frac{d}{dr}v'(1)<0$. Indeed, since $u^*$ is stable, $v$ is stable. It means that $-\Delta- ae^v$ is nonnegative. Thus, we verify that $0$ is the first eigenvalue of $-\Delta-ae^v$ and $v'$ is a first eigenfunction of this operator. Moreover, by Proposition \ref{comparison2}, we have $v'\ge 0$ in $B_{r_h}$. Thus, we get the claim. In particular, we have
\begin{equation*}
    \int_{B_1} a(|x|)e^v {v'}^3\,dx>0.
\end{equation*}
Thus, we get the result.
\vspace{5pt}

\noindent

\textbf{Step 3.} \textit{Conclusion}.

We remark that there exists $h>0$ depending only on $a$ such that \eqref{atarimae} holds. Thus, thanks to Proposition \ref{mainprop}, we know that every solution is nondegenerate if $\beta$ is sufficiently large. Moreover, thanks to Theorem \ref{diagramth}, the bifurcation curve is analytic. Therefore, we verify that the bifurcation diagram is of Type III. 
\end{proof}

\section*{Acknowledgments}
The author would like to thank Associate Professor Michiaki Onodera and Professor Yasuhito Miyamoto, for their valuable advice.
This work was supported by JSPS KAKENHI Grant Number 23KJ0949.
\section{Appendix}
In this section, we introduce some basic properties.  
We first define the symmetric rearrangement. Let $f:\mathbb{R}\to \mathbb{R}$ be a measurable function which decays to zero at infinity. Then, we define the radially symmetric function  
\begin{equation*}
f^*(|x|)= \sup \{t>0:|\{y\in \mathbb{R}^N:|f(y)|>t\}|>w_{N}|x|^N\}
\end{equation*}
as the symmetric rearrangement of $f$, where $|G|$ is the Lebesgue measure of a set $G\in \mathbb{R}^N$ and $w_N$ is the volume of the $N$-dimensional unit ball. Then, it is well-known that the following lemma holds.
\begin{lemma}[see \cite{Lieb,Kaw}]
\label{apenlem1}
    Let $N\ge 3$ and $f\in H^1_{0}(\mathbb{R}^N)$. Then, the symmetric rearrangement $f^*$ satisfies
\begin{equation*}
    \lVert f^* \rVert_{L^2(\mathbb{R}^N)}=\lVert f\rVert_{L^2(\mathbb{R}^N)},\hspace{4mm} \lVert \nabla f^* \rVert_{L^2(\mathbb{R}^N)}\le \lVert f\rVert_{L^2(\mathbb{R}^N)},
\end{equation*}
and
\begin{equation*}
\int_{\mathbb{R}^{N}}\frac{|f^*|^2}{|x|^2}\,dx\ge \int_{\mathbb{R}^{N}} \frac{|f|^2}{|x|^2}\,dx.
\end{equation*}
\end{lemma}
Then, we introduce the following lemma, which is used in Section \ref{separate sec}.
\begin{lemma}
\label{rough lemma}
Assume that $N\ge 2$ and let $R>0$. Moreover, let $w\in C^2(B_{R})$ be a radially symmetric superharmonic function. Then, it follows $w'(r)\le 0$ in $(0,R)$.
\end{lemma}
\begin{proof}
Since $w$ is superharmonic, it holds $(r^{N-1}w'(r))'\le 0$. Suppose the contrary, i.e., we assume that $w'(t)>0$ for some $R>t>0$. Then, we have $t^{N-1}w'(t)\le r^{N-1}w'(r)$ for all $0<r<t$, Since $w\in C^2(B_R)$, we get 
$t^{N-1}w'(t)\le 0$ by letting $r\to 0$. It is a contradiction. Thus, it holds $w'\le 0$.
\end{proof}
Finally, we consider the regularity of radial solution for  
\begin{equation}
\label{apeneq}
-\Delta u = K(|x|) u \hspace{4mm}\text{in $B_1$.} 
\end{equation}

\begin{lemma}
\label{apenlem2}
Let $N\ge 2$ and $K\in C^0(\overline{B_1})$. We assume that  $u\in H^1_{0}(B_1)$ is a radial solution of \eqref{apeneq} in the weak sense. Then, $u\in C^2_{0}(\overline{B_1}).$
\end{lemma}
\begin{proof}
    By the elliptic regularity theory (see \cite{gil}), we have $u\in C^1_{0}(\overline{B_1})$. Since $u$ is radially-symmetric, it holds $\frac{d}{dr} u(0)=0$.
    Moreover, for any $0<r_0<r\le 1$, we have
 \begin{equation*}
    u(r)=u(r_0)+\int^{r}_{r_0}\left(
         \left(\frac{r_0}{s}\right)^{N-1} \frac{du}{ds}(r_0)-
         \int^{s}_{r_0}\left(\frac{t}{s}\right)^{N-1}K(t)u\,dt\right)\,ds.
     \end{equation*}
By letting $r_0\to 0$, we get
\begin{equation*}
         u(r)=u(0)-\int^{r}_{0}\,ds
         \int^{s}_{0}\left(\frac{t}{s}\right)^{N-1}K(t)u\,dt
\end{equation*}
for all $0\le r\le 1$. By using an ODE technique, we get $u\in C^2_{0}(\overline{B_1})$.
\end{proof}

\bibliographystyle{plain}
\bibliography{bifurcation}

\begin{thebibliography}{10}

\bibitem{A2016}
A.~Aghajani.
\newblock New a priori estimates for semistable solutions of semilinear
  elliptic equations.
\newblock {\em Potential Anal.}, 44(4):729--744, 2016.

\bibitem{AC}
A.~Aghajani, C.~Cowan, and A.~Moameni.
\newblock The {G}elfand problem on annular domains of double revolution with
  monotonicity.
\newblock {\em Proc. Amer. Math. Soc.}, 150(8):3457--3470, 2022.

\bibitem{Ali}
Nicholas~D. Alikakos and Peter~W. Bates.
\newblock On the singular limit in a phase field model of phase transitions.
\newblock {\em Ann. Inst. H. Poincar\'{e} Anal. Non Lin\'{e}aire},
  5(2):141--178, 1988.

\bibitem{Bae}
Soohyun Bae.
\newblock Entire solutions with asymptotic self-similarity for elliptic
  equations with exponential nonlinearity.
\newblock {\em J. Math. Anal. Appl.}, 428(2):1085--1116, 2015.

\bibitem{Baeni}
Soohyun Bae and Wei-Ming Ni.
\newblock Existence and infinite multiplicity for an inhomogeneous semilinear
  elliptic equation on {$\mathbf{ R}^n$}.
\newblock {\em Math. Ann.}, 320(1):191--210, 2001.

\bibitem{B}
Ha\"{\i}m Brezis, Thierry Cazenave, Yvan Martel, and Arthur Ramiandrisoa.
\newblock Blow up for {$u_t-\Delta u=g(u)$} revisited.
\newblock {\em Adv. Differential Equations}, 1(1):73--90, 1996.

\bibitem{Br}
Haim Brezis and Juan~Luis V\'{a}zquez.
\newblock Blow-up solutions of some nonlinear elliptic problems.
\newblock {\em Rev. Mat. Univ. Complut. Madrid}, 10(2):443--469, 1997.

\bibitem{Nor}
C.~Budd and J.~Norbury.
\newblock Semilinear elliptic equations and supercritical growth.
\newblock {\em J. Differential Equations}, 68(2):169--197, 1987.

\bibitem{cabre2010}
Xavier Cabr\'{e}.
\newblock Regularity of minimizers of semilinear elliptic problems up to
  dimension 4.
\newblock {\em Comm. Pure Appl. Math.}, 63(10):1362--1380, 2010.

\bibitem{cabre2017}
Xavier Cabr\'{e}.
\newblock Boundedness of stable solutions to semilinear elliptic equations: a
  survey.
\newblock {\em Adv. Nonlinear Stud.}, 17(2):355--368, 2017.

\bibitem{cabre2019}
Xavier Cabr\'{e}.
\newblock A new proof of the boundedness results for stable solutions to
  semilinear elliptic equations.
\newblock {\em Discrete Contin. Dyn. Syst.}, 39(12):7249--7264, 2019.

\bibitem{cabrecapella}
Xavier Cabr\'{e} and Antonio Capella.
\newblock Regularity of radial minimizers and extremal solutions of semilinear
  elliptic equations.
\newblock {\em J. Funct. Anal.}, 238(2):709--733, 2006.

\bibitem{CCS}
Xavier Cabr\'{e}, Antonio Capella, and Manel Sanch\'{o}n.
\newblock Regularity of radial minimizers of reaction equations involving the
  {$p$}-{L}aplacian.
\newblock {\em Calc. Var. Partial Differential Equations}, 34(4):475--494,
  2009.

\bibitem{CFRS}
Xavier Cabr\'{e}, Alessio Figalli, Xavier Ros-Oton, and Joaquim Serra.
\newblock Stable solutions to semilinear elliptic equations are smooth up to
  dimension 9.
\newblock {\em Acta Math.}, 224(2):187--252, 2020.

\bibitem{CR-O}
Xavier Cabr\'{e} and Xavier Ros-Oton.
\newblock Regularity of stable solutions up to dimension 7 in domains of double
  revolution.
\newblock {\em Comm. Partial Differential Equations}, 38(1):135--154, 2013.

\bibitem{chend}
Wenjing Chen and Juan D\'{a}vila.
\newblock Resonance phenomenon for a {G}elfand-type problem.
\newblock {\em Nonlinear Anal.}, 89:299--321, 2013.

\bibitem{chen}
Jann-Long Chern, Zhi-You Chen, Jhih-He Chen, and Yong-Li Tang.
\newblock On the classification of standing wave solutions for the
  {S}chr\"{o}dinger equation.
\newblock {\em Comm. Partial Differential Equations}, 35(2):275--301, 2010.

\bibitem{CR}
Michael~G. Crandall and Paul~H. Rabinowitz.
\newblock Some continuation and variational methods for positive solutions of
  nonlinear elliptic eigenvalue problems.
\newblock {\em Arch. Rational Mech. Anal.}, 58(3):207--218, 1975.

\bibitem{Davdup}
J.~D\'{a}vila and L.~Dupaigne.
\newblock Perturbing singular solutions of the {G}elfand problem.
\newblock {\em Commun. Contemp. Math.}, 9(5):639--680, 2007.

\bibitem{Flore}
Jean Dolbeault and Isabel Flores.
\newblock Geometry of phase space and solutions of semilinear elliptic
  equations in a ball.
\newblock {\em Trans. Amer. Math. Soc.}, 359(9):4073--4087, 2007.

\bibitem{Dup}
Louis Dupaigne.
\newblock {\em Stable solutions of elliptic partial differential equations},
  volume 143 of {\em Chapman \& Hall/CRC Monographs and Surveys in Pure and
  Applied Mathematics}.
\newblock Chapman \& Hall/CRC, Boca Raton, FL, 2011.

\bibitem{Marius}
Marius Ghergu and Olivier Goubet.
\newblock Singular solutions of elliptic equations with iterated exponentials.
\newblock {\em J. Geom. Anal.}, 30(2):1755--1773, 2020.

\bibitem{Gidas}
B.~Gidas, Wei~Ming Ni, and L.~Nirenberg.
\newblock Symmetry and related properties via the maximum principle.
\newblock {\em Comm. Math. Phys.}, 68(3):209--243, 1979.

\bibitem{gil}
David Gilbarg and Neil~S. Trudinger.
\newblock {\em Elliptic partial differential equations of second order}.
\newblock Classics in Mathematics. Springer-Verlag, Berlin, 2001.
\newblock Reprint of the 1998 edition.

\bibitem{Gui}
Changfeng Gui.
\newblock On positive entire solutions of the elliptic equation {$\Delta
  u+K(x)u^p=0$} and its applications to {R}iemannian geometry.
\newblock {\em Proc. Roy. Soc. Edinburgh Sect. A}, 126(2):225--237, 1996.

\bibitem{Guini}
Changfeng Gui, Wei-Ming Ni, and Xuefeng Wang.
\newblock On the stability and instability of positive steady states of a
  semilinear heat equation in {$\mathbf{ R}^n$}.
\newblock {\em Comm. Pure Appl. Math.}, 45(9):1153--1181, 1992.

\bibitem{Guowei}
Zongming Guo and Juncheng Wei.
\newblock Global solution branch and {M}orse index estimates of a semilinear
  elliptic equation with super-critical exponent.
\newblock {\em Trans. Amer. Math. Soc.}, 363(9):4777--4799, 2011.

\bibitem{JL}
D.~D. Joseph and T.~S. Lundgren.
\newblock Quasilinear {D}irichlet problems driven by positive sources.
\newblock {\em Arch. Rational Mech. Anal.}, 49:241--269, 1972/73.

\bibitem{Kaw}
Bernhard Kawohl.
\newblock {\em Rearrangements and convexity of level sets in {PDE}}, volume
  1150 of {\em Lecture Notes in Mathematics}.
\newblock Springer-Verlag, Berlin, 1985.

\bibitem{KiWei}
Hiroaki Kikuchi and Juncheng Wei.
\newblock A bifurcation diagram of solutions to an elliptic equation with
  exponential nonlinearity in higher dimensions.
\newblock {\em Proc. Roy. Soc. Edinburgh Sect. A}, 148(1):101--122, 2018.

\bibitem{korman}
Philip Korman.
\newblock Solution curves for semilinear equations on a ball.
\newblock {\em Proc. Amer. Math. Soc.}, 125(7):1997--2005, 1997.

\bibitem{Luo}
Baishun Lai and Qing Luo.
\newblock Uniqueness of singular solution of semilinear elliptic equation.
\newblock {\em Proc. Indian Acad. Sci. Math. Sci.}, 120(5):583--591, 2010.

\bibitem{Lieb}
Elliott~H. Lieb and Michael Loss.
\newblock {\em Analysis}, volume~14 of {\em Graduate Studies in Mathematics}.
\newblock American Mathematical Society, Providence, RI, second edition, 2001.

\bibitem{Lin}
Song~Sun Lin.
\newblock Positive singular solutions for semilinear elliptic equations with
  supercritical growth.
\newblock {\em J. Differential Equations}, 114(1):57--76, 1994.

\bibitem{Merle}
F.~Merle and L.~A. Peletier.
\newblock Positive solutions of elliptic equations involving supercritical
  growth.
\newblock {\em Proc. Roy. Soc. Edinburgh Sect. A}, 118(1-2):49--62, 1991.

\bibitem{Mi2014}
Yasuhito Miyamoto.
\newblock Structure of the positive solutions for supercritical elliptic
  equations in a ball.
\newblock {\em J. Math. Pures Appl. (9)}, 102(4):672--701, 2014.

\bibitem{Mi2015}
Yasuhito Miyamoto.
\newblock Classification of bifurcation diagrams for elliptic equations with
  exponential growth in a ball.
\newblock {\em Ann. Mat. Pura Appl. (4)}, 194(4):931--952, 2015.

\bibitem{Mi2018}
Yasuhito Miyamoto.
\newblock A limit equation and bifurcation diagrams of semilinear elliptic
  equations with general supercritical growth.
\newblock {\em J. Differential Equations}, 264(4):2684--2707, 2018.

\bibitem{Mi2020}
Yasuhito Miyamoto and Y\={u}ki Naito.
\newblock Fundamental properties and asymptotic shapes of the singular and
  classical radial solutions for supercritical semilinear elliptic equations.
\newblock {\em NoDEA Nonlinear Differential Equations Appl.}, 27(6):Paper No.
  52, 25, 2020.

\bibitem{Mi2023}
Yasuhito Miyamoto and Y\={u}ki Naito.
\newblock Singular solutions for semilinear elliptic equations with general
  supercritical growth.
\newblock {\em Ann. Mat. Pura Appl. (4)}, 202(1):341--366, 2023.

\bibitem{Ned}
Gueorgui Nedev.
\newblock Regularity of the extremal solution of semilinear elliptic equations.
\newblock {\em C. R. Acad. Sci. Paris S\'{e}r. I Math.}, 330(11):997--1002,
  2000.

\bibitem{Niserrin}
Wei-Ming Ni and James Serrin.
\newblock Nonexistence theorems for singular solutions of quasilinear partial
  differential equations.
\newblock {\em Comm. Pure Appl. Math.}, 39(3):379--399, 1986.

\bibitem{sansan}
Manel Sanch\'{o}n.
\newblock Boundedness of the extremal solution of some {$p$}-{L}aplacian
  problems.
\newblock {\em Nonlinear Anal.}, 67(1):281--294, 2007.

\bibitem{Vil}
Salvador Villegas.
\newblock Boundedness of extremal solutions in dimension 4.
\newblock {\em Adv. Math.}, 235:126--133, 2013.

\bibitem{Ye}
Dong Ye and Feng Zhou.
\newblock Boundedness of the extremal solution for semilinear elliptic
  problems.
\newblock {\em Commun. Contemp. Math.}, 4(3):547--558, 2002.

\end{thebibliography}

\end{document}